\documentclass[a4paper,10pt]{article}

\usepackage{graphicx}
\usepackage{amsmath}
\usepackage{amsfonts}
\usepackage{titlesec}
\usepackage{amsthm}
\usepackage{amssymb}

\usepackage{color}

\newcommand{\nablax}{\nabla_{\hspace{-2pt}x}\hspace{1pt}}
\newcommand{\nablay}{\nabla_{\hspace{-2pt}y}\hspace{1pt}}
\newcommand{\Tau}{\mathcal{T}}
\newcommand{\Tauh}{\mathcal{T}_h}

\newcommand{\R}{\mathbb{R}}

\newcommand{\D}{\mathcal{D}}

\newcommand{\Epsilon}{\mathcal{E}}

\newtheorem{theorem}{Theorem}[section]
\newtheorem{corollary}{Corollary}

\newtheorem{lemma}[theorem]{Lemma}

\theoremstyle{definition}
\newtheorem{definition}[theorem]{Definition}
\newtheorem{remark}{Remark}
\newtheorem{assumption}[theorem]{Assumption} %
\newtheorem{conclusion}[theorem]{Conclusion} %

\allowdisplaybreaks[1]
\usepackage{fancyheadings}

\usepackage{geometry}
\begin{document}

\renewcommand{\thefootnote}{\fnsymbol{footnote}}
\pagenumbering{arabic}
\pagestyle{myheadings}\markright{}
\titleformat{\section}{\normalfont\normalsize\bfseries}{\thesection}{1em}{}
\titleformat{\subsection}{\normalfont\normalsize\scshape}{\thesubsection}{1em}{}
\begin{center}{\bf \large ADAPTIVE HETEROGENEOUS MULTISCALE METHODS FOR IMMISCIBLE TWO-PHASE FLOW IN POROUS MEDIA\footnote{This work was supported by the Deutsche Forschungsgemeinschaft under the contract numbers SCHW 639/3-2 and OH 98/4-2.}}
\end{center}
\renewcommand{\thefootnote}{\arabic{footnote}}
\begin{center}
{\sc \small PATRICK HENNING{$\hspace{0pt}^1$},
MARIO OHLBERGER\footnote[1]{Westf\"alische Wilhelms-Universit\"at M\"unster, Institut f\"ur Numerische und Angewandte Mathematik, Einsteinstr. 62, D-48149 M\"unster, Germany} AND BEN SCHWEIZER\footnote[2]{Technische Universit\"at Dortmund, Fakult\"at f\"ur Mathematik, Vogelpothsweg 87, D-44227 Dortmund, Germany}}
\end{center}

\begin{abstract}
In this contribution we present the first formulation of a heterogeneous
multiscale method for an incompressible immiscible two-phase flow system with
degenerate permeabilities. The method is in a general formulation which includes
oversampling. We do not specify the discretization of the derived 
macroscopic equation, but we give two examples of possible realizations, 
suggesting a finite element solver for the fine scale and a vertex centered finite volume method for the effective coarse scale equations. Assuming periodicity, we show that the method is
equivalent to a discretization of the homogenized equation. We provide an
a-posteriori estimate for the error between the homogenized solutions of the pressure
and saturation equations and the corresponding HMM approximations. The error
estimate is based on the results recently achieved in [C.~Canc{\`e}s, I.~S. Pop, and M.~Vohral\'{\i}k.
An a posteriori error estimate for vertex-centered finite volume
discretizations of immiscible incompressible two-phase flow.
{\em Math. Comp.}, 2014].
\end{abstract}

\paragraph*{Keywords}
adaptivity, HMM, multiscale problem, two-phase flow, porous media

\paragraph*{AMS subject classifications}
76S05, 35B27, 65G99, 65N08, 65N30

\section{Introduction}

In this paper, we consider degenerate two-phase flow in porous media for immiscible and incompressible fluids with highly variable porosity and a highly variable intrinsic permeability. Applications are oil reservoir simulations (with fluids oil and water) or carbon sequestration (with fluids water and liquid carbon dioxide). The numerical treatment of the two-phase flow system with classical finite element (FE) or finite volume (FV) schemes is typically quite expensive. In particular, when it is necessary to resolve the fine scale features of the (rapidly oscillating) hydraulic conductivity, an extremely fine computational mesh is required. In many scenarios this results in a tremendous computational load that cannot be handled by available computers. To overcome this issue, the classical scheme can be combined with a so called multiscale approach that splits the complex global problem into smaller pieces that can be treated independently. In this work we propose and analyze a strategy how this can be accomplished for the two-phase flow system
on the basis of a two scale formulation derived from the underlying fine scale problem.

There is a vast literature on two-phase flow in porous media. Concerning existence and uniqueness of weak solutions we refer to \cite{Kroener:Luckhaus:1984,Chen:2001,Chen:2002}. Various numerical methods have been introduced and analyzed. However, rigorous convergence results are still rare. Convergence of a 'mixed finite element - finite volume method' was shown in \cite{Ohlberger:1997}, a-priori error estimates for a 'mixed finite element - Galerkin finite element method' were derived in \cite{Chen:Ewing:2001}, convergence of a finite volume scheme based on the Kirchhoff transformed system was obtained in \cite{Michel:2003} and convergence of a finite volume phase-by-phase upstream weighting approach was studied in \cite{Eymard:Herbin:Michel:2003}. For a-priori error estimates for a $hp$ discontinuous Galerkin method we refer to \cite{Epshteyn:Riviere:2009}. Furthermore, rigorous a-posteriori error estimation for finite volume discretizations of incompressible two-phase flow equations were recently derived in \cite{Cances:Pop:Vohralik:2012}.

In the case of rapidly oscillating coefficient functions as considered in this contribution, the classical methods (such as FVM, FEM, $hp$-FEM, DG-FEM or mixed FEM) are often too expensive and should be combined with a multiscale approach. There are typically two ways. One way is that the effective (or upscaled) macroscopic features of the coefficients are determined in a preprocess to apply the classical scheme to the new 'effective equation' (cf. \cite{Abdulle:2005,E:Engquist:2003,Gloria:2006,Henning:Ohlberger:2011_2}). The other way is that the multiscale features are used to assemble a suitable set of 'multiscale basis functions' that replaces the original set of basis functions in a classical scheme (cf. \cite{BFH97,Hou:Wu:1997,MaP14,OZB13}).

In this work, we are dealing with the first way of approaching the problem. To state the two-phase flow system, let us denote
\begin{itemize}
\item $\Omega \subset \mathbb{R}^d$, $d=2,3$ an open, bounded domain with
a polyhedral Lipschitz boundary,
\item $\Omega_T:= \Omega \times (0,T]$, where $(0,T] \subset \R_{>0}$ denotes a time interval, 
\item $\Phi^{\epsilon}(x)$ the porosity in $\Omega$ and $K^{\epsilon}(x)$ the
absolute permeability in $\Omega$,
\item $s_w^{\epsilon}(x,t)$, $s_n^{\epsilon}(x,t)$ the saturations and
$p_w^{\epsilon}(x,t)$, $p_n^{\epsilon}(x,t)$  the pressures of wetting and
non-wetting phase respectively,
\item $P_c(s)$ the capillary pressure function (only depending on the
saturation),
\item $k_{r,w}(s)$, $k_{r,n}(s)$ the relative permeabilities (only depending on the
saturation),
\item $\mu_{w}$, $\mu_{n} \in \mathbb{R}_{>0}$ the viscosities of wetting and
non-wetting phase,
\item $\rho_{w}$, $\rho_{n} \in \mathbb{R}_{>0}$ the densities of wetting and
non-wetting fluid,
\item $g \in \mathbb{R}^d$ the downward pointing gravity vector.
\end{itemize}
With these definitions the classical two-phase flow system can be characterized by the following system of six equations. The two equations for the mass balance are given by
\begin{align}
\label{classical-two-phase-system}\Phi^{\epsilon} \partial_t
s^{\epsilon}_{\alpha} + \nabla \cdot u_{\alpha}^{\epsilon} &= 0 \quad \mbox{in}
\enspace \Omega_T,
\end{align}
where the subindex $\alpha\in \{ w, n\}$ stands either
for the wetting phase $w$ or the
non-wetting phase $n$. The flux $u_{\alpha}^{\epsilon}$ is given by the Darcy-Muskat law
\begin{align*}
u_{\alpha}^{\epsilon} = - K^{\epsilon}
\frac{k_{r,\alpha}(s_{\alpha}^{\epsilon})}{\mu_{\alpha}} ( \nabla
p_{\alpha}^{\epsilon} - \rho_{\alpha} g ),
\end{align*}
again for $\alpha\in \{ w, n\}$.
Additionally, we have an algebraic coupling between these equations: the
saturations sum up to $1$ and the difference between the two pressures is
given by the monotonically decreasing capillary pressure function:
\begin{align*}
s^{\epsilon}_w + s^{\epsilon}_n = 1 \quad \mbox{and} \quad P_c(s_w^{\epsilon})=
p_{n}^{\epsilon} - p_{w}^{\epsilon} \enspace \mbox{with} \enspace
P_c^{\prime}(s)<0.
\end{align*}
In order to avoid a numerical scheme that needs to resolve the full microstructure in (\ref{classical-two-phase-system}), which would be far too expensive, we replace the above system by a homogenized/upscaled (i.e. Õoscillation-freeÕ) system of the structure
\begin{align}
\label{initial-upscaled-problem}\Phi^0 \partial_t s^0_{\alpha} - \nabla \cdot \left( K^0_{\epsilon,\kappa} \frac{k_{r,\alpha}(s_{\alpha}^0)}{\mu_{\alpha}} ( \nabla p_{\alpha}^0 - \rho_{\alpha} g ) \right) &= 0 \quad \mbox{in} \enspace \Omega_T, \enspace \mbox{for} \enspace \alpha=w,n,\\
\nonumber s^0_w + s^0_n = 1 \quad \mbox{and} \quad P_c(s_w^0)&= p_{n}^0 - p_{w}^0,
\end{align}
where $\Phi^0$ is a local average of $\Phi^{\epsilon}$ and $K^0_{\epsilon,\kappa}$ is defined by
\begin{align*}
\left(K^0_{\epsilon,\kappa}\right)_{ij}(x) := \int_{Y_{\kappa_0,\kappa}} \hspace{-26pt}-\hspace{7pt} K^{\epsilon}(x+ \kappa y) (e_i + \nablay w_{\kappa}^i(x,y))\cdot  (e_j + \nablay w_{\kappa}^j(x,y)) \hspace{2pt} dy,
\end{align*}
with local representative cells $Y_{\kappa_0,\kappa}$ and $w_{\kappa}^j$ suitable solutions of (cheaply) pre-computed local problems that contain micro-structure information. The upscaled system (\ref{initial-upscaled-problem}) can be solved with feasible cost with any favorite classical solver for the two-phase flow equations. In this sense, our strategy gives rise to a large class of multiscale methods for the two-phase flow system, depending on the classical method that it is combined with. The original idea behind this approach goes back to the heterogeneous multiscale finite element method (HMM) initially introduced in \cite{E:Engquist:2003}. It is known that this method is very reliable in many applications and that it can be interpreted as a discretization of an upscaled problem of the type above (cf. \cite{Gloria:2006,Gloria:2008}). In the subsequent sections, we will give a detailed motivation and justification for the upscaled/homogenized system (\ref{initial-upscaled-problem}). For a priori and a posteriori error estimates of HMM approximations to  elliptic problems we refer to \cite{Abdulle:2005,Abdulle:2009,Abdulle:E:Engquist:Vanden-Eijnden:2012,Abdulle:Nonnenmacher:2011,Henning:Ohlberger:2009,Henning:Ohlberger:2011_2,Ohlberger:2005}. Homogenized problems as (\ref{initial-upscaled-problem}) can be sometimes rigorously justified as limit problems for the case that the characteristic length scale $\epsilon$ of the microscopic oscillations converges to zero. For these kind of homogenization results for the immiscible incompressible two-phase flow equations we refer to \cite{Bourgeat:Hidani:1995,Bourgeat:Luckhaus:Mikelic:1996,Yeh:2006,Henning:Ohlberger:Schweizer:2012} and the references therein.

There are several contributions proposing other types of multiscale methods for the two-phase flow system. However, each of these methods is formulated under the assumption that the capillary pressure can be neglected and consequently that the pressures of the two phases are identical. This is an assumption that we do not make. In the following, we give a survey on multiscale methods for two-phase flow, however, with the remark that the topic is too comprehensive to name all of these schemes.  A local-global upscaling for two-phase flow was stated in \cite{Chen:Li:2009} and an adaptive variational multiscale method (VMM) for multiphase flow was proposed in \cite{Nordbotten:2008}. Furthermore, there are several methods that fit into the framework of the Multiscale Finite Element Method (MsFEM) proposed in \cite{Hou:Wu:1997}) or its modified version, the Mixed Multiscale Finite Element Method, proposed in \cite{Chen:Hou:2003}. Multiscale methods based on a Finite Volume or Finite Volume Element approach can be for instance found in \cite{Lunati:Lee:2009,Lunati:Jenny:2008,Jenny:Lee:Tchelep:2004,Jiang:Mishev:2012,Aarnes:Lie:Kippe:Krogstad:2009,Ginting:2004,Jenny:Lee:Tchelep:2003}. The construction of conservative fluxes in a post-processing step for Generalized Multiscale Finite Element approximations (GMsFEM) was suggested in \cite{Bush:Ginting:Presho:2014}. 
Furthermore, we refer to the hierarchical multiscale method for two-phase flow based on mixed finite elements \cite{Aarnes:Krogstad:Lie:2006}, an adaptive multiscale method also based on mixed finite elements \cite{Aarnes:Efendiev:2006} and a modified multiscale method based on a finite volume scheme on the coarse scale \cite{Efendiev:Ginting:Hou:Ewing:2006}. Finally, a Galerkin- and a mixed multiscale finite element method are analyzed in \cite{Jiang:Aarnes:Efendiev:2012}, however ignoring gravity and capillary effects.

The advantage of our approach compared to other multiscale methods is that the microstructure does not have to be resolved globally, but only in a small set of cubic environments of quadrature points. Hence, the new method can operate far below linear computational complexity in each iteration step, whereas the above mentioned approaches can only at best operate with linear complexity, but never below. Hence, previously uncomputable problems may become computable now.


\section{Weak formulation of the  fine scale two-phase flow equations}
\label{sec2}

In this section we state a weak formulation of the incompressible two-phase flow equations after Kirchhoff transformation and recall a corresponding existence and uniqueness result. For this purpose let us define the phase mobility functions by
\begin{align*}
\lambda_{w}(s) := \frac{k_{r,w}(s)}{\mu_{w}}, \quad \lambda_{n}(s) :=
\frac{k_{r,n}(1-s)}{\mu_{n}} \quad \mbox{and} \quad  \lambda(s) :=  \lambda_{w}(s)
+ \lambda_{n}(s).
\end{align*}
Typically, $\lambda_w$ is an increasing function with $\lambda_w(0)=0$ and
$\lambda_n$ is a decreasing function with $\lambda_n(1)=0$.

%
%
In this setting we can use the Kirchhoff transformation as a mathematical trick to rewrite the problem in terms of a global pressure $P^{\epsilon}$ (c.f.
\cite{Amaziane:et-al:2010,Antontsev:Kazhikhov:Monakhov:1990,Chen:2001}) which allows to formulate a new problem with a strictly positive total mobility $\lambda(s^{\epsilon}_w)$.
Defining
\begin{align}
\label{global-pressure-def}S^{\epsilon}:= s_w^{\epsilon} \quad \mbox{and} \quad
P^{\epsilon} := p_w^{\epsilon} + \int_0^{S^{\epsilon}}
\frac{\lambda_n(s)}{\lambda(s)} P_c^{\prime}(s) \hspace{2pt} ds
\end{align}
we get the desired relation
\begin{align*}
\lambda(S^{\epsilon}) \nabla P^{\epsilon} &= \lambda_w(S^{\epsilon}) \nabla p_w^{\epsilon} + \lambda_n(S^{\epsilon}) \nabla p_n^{\epsilon}.
\end{align*}
Using the Kirchhoff transform $\Upsilon:[0,1] \to \R$ given by
\begin{align}
\label{kirchhoff-transform-def}\Upsilon^{\epsilon}:=\Upsilon(S^{\epsilon}):= -\int_0^{S^{\epsilon}}
\frac{\lambda_w(s)\lambda_n(s)}{\lambda(s)} P_c^{\prime}(s) \hspace{2pt} ds,
\end{align}
we get (cf. \cite{Chavent:Jaffre:1986})  the following (Kirchhoff transformed) two-phase flow system:
\begin{align}
\label{classical-global-pressure-system} - \nabla \cdot \left( K^{\epsilon}
(\lambda(S^{\epsilon}) \nabla P^{\epsilon} - (\lambda_w(S^{\epsilon}) \rho_w +
\lambda_n(S^{\epsilon}) \rho_n) g )\right) &= 0 \\
\nonumber \Phi^{\epsilon} \partial_t S^{\epsilon} - \nabla \cdot \left(
K^{\epsilon} (\lambda_{w}(S^{\epsilon}) \nabla P^{\epsilon} + \nabla
\Upsilon^{\epsilon} - \lambda_w(S^{\epsilon}) \rho_w g) \right) &= 0.
\end{align}
Thanks to the assumptions on $\lambda_{n}$, $\lambda_{w}$ and $P_c$, the Kirchhoff transform $\Upsilon$ defined in (\ref{kirchhoff-transform-def}) is a strictly increasing function on $[0,1]$.

From now on, we assume inhomogeneous Dirichlet boundary conditions for both
saturation and pressure:
\begin{align*}
S^{\epsilon} = \bar{S} \enspace \mbox{and} \enspace P^{\epsilon} = \bar{P} \quad
\mbox{on} \enspace \partial \Omega \times (0,T].
\end{align*}
Properties of the boundary  functions $\bar{S}$ and $\bar{P}$ are specified below. An initial condition is given by
\begin{align*}
S^{\epsilon}(\cdot,0) = S_0 \quad \mbox{in} \enspace \Omega, \quad \mbox{for a function $S_0$ with } 0 \le S_0 \le 1, \mbox{ a.e. in } \Omega. 
\end{align*}
We note that results of this contribution can be easily generalized to mixed Dirichlet/Neumann boundary
conditions. We now state a weak formulation
of the Kirchhoff transformed two-phase flow system
(\ref{classical-global-pressure-system}). In the following, we will work with this weak formulation.

\begin{definition}[Weak formulation of the two-phase flow system]
We define the solution space by
\begin{align}
\label{solution-space}\Epsilon := \{ (S,P)\hspace{2pt} | \hspace{2pt} &S \in
C^0([0,T],L^2(\Omega)), \hspace{2pt} \Phi^{\epsilon} \partial_t S \in
L^2((0,T),H^{-1}(\Omega)),\\
\nonumber&\Upsilon(S)-\Upsilon(\bar{S}) \in L^2((0,T),H^1_0(\Omega)),
\hspace{2pt}P-\bar{P} \in L^2((0,T),H^1_0(\Omega)) \}.
\end{align}
We call $(S^{\epsilon},P^{\epsilon}) \in \Epsilon$ a {\it weak solution} of the Kirchhoff transformed
two-phase flow system (\ref{classical-global-pressure-system}) 
to initial condition $S_0\in L^\infty(\Omega)$, if
\begin{eqnarray}
\label{weak-form-1}\lefteqn{\int_{0}^T \langle \Phi^{\epsilon} \partial_t
S^{\epsilon}(\cdot,t), \Psi(\cdot,t ) \rangle_{H^{-1}(\Omega),H^1_0(\Omega)}
\hspace{2pt} dt}\\
\nonumber&=& - \int_{\Omega_T} K^{\epsilon} (\lambda_{w}(S^{\epsilon}) \nabla
P^{\epsilon} + \nabla \Upsilon(S^{\epsilon}) - \lambda_w(S^{\epsilon}) \rho_w
g) \cdot \nabla \Psi  \hspace{2pt} dt  \hspace{2pt} dx\\
\label{weak-form-2}0 &=& \int_{\Omega_T} K^{\epsilon} ( \lambda(S^{\epsilon})
\nabla P^{\epsilon} - (\lambda_w(S^{\epsilon}) \rho_w + \lambda_n(S^{\epsilon})
\rho_n) g ) \cdot \nabla \Psi \hspace{2pt} dt  \hspace{2pt} dx
\end{eqnarray}
for all $\Psi \in L^2((0,T),H^1_0(\Omega))$ and if $S^{\epsilon}(\cdot,0)=S_0$.
\end{definition}
The weak formulation was proposed by Chen \cite{Chen:2001}, where the above
formulation already incorporates the regularity results that he derived in
Section 4 of the mentioned work. We emphasize that the Kirchhoff transformed two-phase flow equations do no longer have an explicit physical meaning. Under the following assumptions the problem
(\ref{weak-form-1})-(\ref{weak-form-2}) has a unique weak solution (c.f.
\cite{Cances:Pop:Vohralik:2012,Chen:2001}).
\begin{assumption}[For the existence and uniqueness of a weak solution]
We make the following assumptions:
\begin{enumerate}
\item[(A1)] $\lambda_{w}, \lambda_{n} \in C^0[0,1]$,
$\lambda_w(0)=\lambda_n(1)=0$, $\lambda_w(s)>0$ for $s>0$, $\lambda_n(s)>0$ for
$s<1$ and there exist positive constants $c_{\lambda}$ and $C_{\lambda}$ so that
for all $s\in[0,1]$:
\begin{align*}
c_{\lambda} \le \lambda(s) \le C_{\lambda} \quad \mbox{and} \quad \lambda_w(s)
\le C_{\lambda}.
\end{align*}
\item[(A2)] $K^{\epsilon} \in [L^{\infty}(\R^d, \R^{d \times d})]$ with $K^{\epsilon}(x)$ symmetric
for every $x$. There exist positive constants $\alpha, \beta > 0$ such that, for every $\epsilon$ and almost every $x \in \R^d$
\begin{align*}
\alpha |\xi|^2 \le K^{\epsilon}(x) \xi \cdot \xi \le \beta |\xi|^2 \quad \forall
\xi \in \mathbb{R}^d.
\end{align*}
\item[(A3)] $\Phi^{\epsilon} \in L^{\infty}(\Omega)$ and
$0<\phi^*\le\Phi^{\epsilon}\le\Phi^*<1$ for all $\epsilon>0$ and a.e. in
$\Omega$. For simplicity, we assume that $\Phi^{\epsilon}$ is ergodic in the sense that there exists a constant $\Phi^0 \in \mathbb{R}_{>0}$ (independent of $\epsilon$) such that $(\Phi^{\epsilon},1)_{L^2(Y_{\epsilon})}=\Phi^0$ for all cubes $Y_{\epsilon}\subset \mathbb{R}^d$ with edge-length $\epsilon$.
\item[(A4)] The Kirchhoff transform $\Upsilon$ is Lipschitz-continuous and
strictly increasing on $[0,1]$.
\item[(A5)] $\bar{P} \in L^{\infty}((0,T),H^{\frac{1}{2}}(\partial \Omega))$
(which implies the existence of an $L^{\infty}((0,T),H^{{1}}(\Omega))$-extension
that we denote again by $\bar{P}$).
\item[(A6)] $\bar{S} \in L^{\infty}(\partial \Omega \times (0,T) )$ with $0 \le
\bar{S} \le 1$. Furthermore, we assume that $\bar{S}$ can be extended to a measurable function on $\Omega_T$ so that:
\begin{align*}
\partial_t \bar{S} \in L^1(\Omega_T), \quad \Upsilon(\bar{S}) \in
L^2((0,T),H^1(\Omega)), \quad \bar{S}(\cdot,0)=S_0.
\end{align*}
\item[(A7)] $S_0 \in L^{\infty}(\Omega)$ with $0 \le S_0 \le 1$ a.e. in $\Omega$.
\item[(A8)] There exists a positive constants $C_U$ such that
\begin{align*}
&(\lambda(s_2)-\lambda(s_1))^2 + \sum_{\alpha=w,n}
(\lambda_{\alpha}(s_2)-\lambda_{\alpha}(s_1))^2 \\
&\quad \le C_U (s_2 - s_1)(\Upsilon(s_2)-\Upsilon(s_1)) \quad \forall
s_1,s_2\in[0,1].
\end{align*}
\item[(A9)] The solution has the regularity $\nabla P^{\epsilon} \in L^{\infty}(\Omega_T)$.
\end{enumerate}
\end{assumption}
\begin{remark}
Assumption (A2) states that $K^{\epsilon}$ is defined on the whole $\R^d$ instead of only on $\Omega$. This simplifies the formalism in later considerations. Actually, it is only required that $K^{\epsilon}$ is reasonably defined on $\Omega$ and a small environment around. All other values of $K^{\epsilon}$ will not have an influence. 

Assumptions (A8) and (A9) are only required to obtain uniqueness of
the weak solution (c.f. Theorem 3.1. in \cite{Chen:2001}). Furthermore, condition
(A9) could be replaced by assumptions on the domain and the data 
which guarantee the Lipschitz-continuity of the pressure $p(\cdot,t)$
almost everywhere in $t$ (see \cite{Chen:2001,Chen:2002}). For simplicity and to
clarify the presentation we directly work with (A9).

Also recall that, in addition to (A1)-(A9), we made the assumption that $P_c$, $k_{r,w}$ and $k_{r,n}$ only depend on the saturation but not on the space variable $x$. This  is a characterization for the case that the relevant computational domain $\Omega$ is occupied by only one type of soil. In the case of various types of soil, $P_c(x,s)$, $k_{r,w}(x,s)$ and $k_{r,n}(x,s)$ are piecewise constant with respect to $x$. However, we also note that this can be a restrictive assumption in many practical applications, since it is known that there is a relation between spatial variations of $\Phi^{\epsilon}$ and $K^{\epsilon}$ on the one hand, and $P_c$ on the other hand (cf. \cite{Chen:Huan:Ma:2006} and the references therein and in particular the classical work by Leverett and Lewis \cite{Leverett:Lewis:1941}).
\end{remark}

\section{The continuous HMM-type two-scale problem}
\label{section-two-scale-problem}

In this section we motivate and state a two-scale problem that should be seen as the continuous limit of a (discrete) heterogeneous multiscale method for the degenerate two phase flow equations. This problem will be the first step towards the final (previously mentioned) upscaled problem (\ref{initial-upscaled-problem}). In the case of $\epsilon$-periodic coefficient functions, we show that the two-scale problem can be characterized as the limit problem of the exact equation for $\epsilon \rightarrow 0$. Hence, it is justified. 
In the following, we denote by $Y:=(-\frac{1}{2},\frac{1}{2})^d$ the $0$-centered unit cube and we define 
$C^1_{\sharp}(\bar{Y})$ as the space of functions on $Y$, for which the $Y$-periodic extension is continuously differentiable.
The space of periodic $H^1$-functions with zero average is then obtained as
\begin{align*}
\tilde{H}^1_{\sharp} (Y) &:= \overline{ \left\{ v \in C^1_\sharp(\bar{Y})\left| \int_{Y}\hspace{-13pt}-\hspace{7pt} v(y) \enspace dy = 0 \right.\right\}}^{\|\cdot\|_{H^{1}(Y)}}.
\end{align*}

\subsection{Motivation and formulation of the method}
\label{subsection-motivation}

The central idea of the Heterogeneous Multiscale Method (HMM) is rather simple: replace the original (expensive-to-solve) fine-scale problem, by an upscaled (cheap-to-solve) problem which is determined by sampling in local cells. In this context, the upscaled properties in the new equation are nothing but averages of the information that was gained from the 'cell problems'. This strategy was initially formulated in a fully discrete setting, in the sense that the sampling cells where related to a fixed set of quadrature points in a coarse mesh. However, as demonstrated in several contributions (see e.g. \cite{Gloria:2006,Gloria:2008,Henning:Ohlberger:2011_2}) it can be also interpreted (and hence formulated) on a fully continuous level, i.e. as an analytical equation. In specific scenarios, this analytical equation coincides with the classical homogenized equation (cf. \cite{Ohlberger:2005}).

For elliptic problems of the type
\begin{align}
\label{fine-scale-elliptic}\mbox{find } u^{\epsilon} \in H^1_0(\Omega): \qquad \int_{\Omega} K^{\epsilon} \nabla u^{\epsilon} \cdot \nabla v =  \int_{\Omega} f v \qquad \mbox{for all } v \in H^1_0(\Omega),
\end{align}
the HMM strategy (on the continuous level) can be summarized as follows. First, for some small fixed parameter $\kappa \ge \epsilon$, define a 'cell basis', which is simply a set of functions that contains sampled information from small cells.
\begin{definition}[Cell basis]
\label{definition-cell-basis}
Let $w_{\kappa}^i \in L^2(\Omega,\tilde{H}^1_{\sharp}(Y))$ solve the problem
\begin{align}
\label{definition-cell-basis-eqn} \int_Y K^{\epsilon}(x+ \kappa y) (e_i + \nablay w_{\kappa}^i(x,y)) \cdot \nablay \psi(y)  \hspace{2pt} dy = 0 
\end{align}
for all $\psi \in \tilde{H}^1_{\sharp}(Y)$. The set
$
\{w_{\kappa}^i \in L^2(\Omega,\tilde{H}^1_{\sharp}(Y))| \enspace 1 \le i \le d\}
$
is called the {\it cell basis}.
\end{definition}
Observe that the sampling domain is a cubic cell of diameter $\kappa$ centered around $x$. Also observe the similarity to cell problems arising in homogenization theory (cf. \cite{Allaire:1992}).

Based on this cell basis, we define the entries of the effective intrinsic permeability $K^0_{\epsilon,\kappa}$ by
\begin{align}
\label{effective-hydraulic-conductivity}\left(K^0_{\epsilon,\kappa}\right)_{ij}(x) := \int_Y K^{\epsilon}(x+ \kappa y) (e_i + \nablay w_{\kappa}^i(x,y))\cdot  (e_j + \nablay w_{\kappa}^j(x,y)) \hspace{2pt} dy.
\end{align}
It is well known, that $K^0_{\epsilon,\kappa}$ has the same spectral bounds as $K^{\epsilon}$, hence it is a coercive and bounded matrix. With that, we can solve the upscaled elliptic problem:
\begin{align}
\label{upscaled-elliptic}\mbox{find } u^0_{\epsilon,\kappa} \in H^1_0(\Omega): \qquad \int_{\Omega} K^0_{\epsilon,\kappa}\nabla u^0_{\epsilon,\kappa} \cdot \nabla v =  \int_{\Omega} f v \qquad \mbox{for all } v \in H^1_0(\Omega).
\end{align}
Observe that problem (\ref{upscaled-elliptic}) is purely macroscopic and can be solved with any standard method on a coarse grid. Hence, it is extremely cheap to solve, compared to the original problem (\ref{fine-scale-elliptic}).

It is possible to clearly justify that $u^0_{\epsilon,\kappa}$ is a good $L^2$-approximation to $u^{\epsilon}$, by the following result by Gloria \cite{Gloria:2006}: if $K^{\epsilon}$ represents a $G$-convergent sequence (see Definition \ref{def-G-convergence} below), then 
\begin{align*}
\lim_{\kappa \rightarrow 0} \lim_{\epsilon \rightarrow 0} \| u^0_{\epsilon,\kappa} - u^{\epsilon} \|_{L^2(\Omega)} = 0.
\end{align*}
This means that problem (\ref{upscaled-elliptic}), with the upscaled coefficient $K^0_{\epsilon,\kappa}$, yields typically highly accurate approximations, independent of structural assumptions on $K^{\epsilon}$.

The above used notion of $G$-convergence (c.f. \cite{Jikov:Kozlov:Oleinik:1994}) is typically used to describe general homogenization settings.
\begin{definition}[$G$-convergence]\label{def-G-convergence}
Let $\Omega \subset \R^d$ denote a bounded domain and let $(A^{\epsilon})_{\epsilon>0} \subset [L^{\infty}(\Omega)]^{d \times d}$ denote a sequence of symmetric matrices that are uniformly bounded and coercive, i.e. there exist $a_0,a_1 \in \R_{>0}$ such that for a.e. $x\in \Omega$:
\begin{align*}
a_0 |\xi|^2 \le A^{\epsilon}(x) \xi \cdot \xi \le a_1 |\xi|^2 \qquad \forall \xi \in \R^d.
\end{align*}
Then, we call $A^{\epsilon}$ $G$-convergent to $A^0 \in [L^{\infty}(\Omega)]^{d \times d}$ if for any $f \in H^{-1}(\Omega)$ the solutions $u^{\epsilon}\in H^1_0(\Omega)$ of
\begin{align*}
\int_{\Omega} A^{\epsilon} \nabla u^{\epsilon} \cdot \nabla v = f(v) \qquad \forall v \in H^1_0(\Omega)
\end{align*}
satisfy the relations
$u^{\epsilon} \rightharpoonup u^0 \enspace \mbox{in} \enspace H^1_0(\Omega) \quad \mbox{and} \quad A^{\epsilon} \nabla u^{\epsilon}\rightharpoonup A^0 \nabla u^0 \enspace \mbox{in} \enspace [L^2(\Omega)]^d,$
where $u^0 \in H^1_0(\Omega)$ is the solution of 
\begin{align*}
\int_{\Omega} A^0 \nabla u^0 \cdot \nabla v = f(v) \qquad \forall v \in H^1_0(\Omega).
\end{align*} 
\end{definition}
We note that the $G$-limit $A^0$ is bounded and coercive with the same spectral bounds as $A^{\epsilon}$ and that the usage of the upscaled coefficient $K^0_{\epsilon,\kappa}$ is well-established and well-justified for elliptic homogenization problems. For the two-phase flow equations, we now suggest to follow the completely same strategy, meaning that we solve (\ref{definition-cell-basis-eqn}), define $K^0_{\epsilon,\kappa}$ according to (\ref{effective-hydraulic-conductivity}) and simply let it replace $K^{\epsilon}$ in the original two-phase flow problem. This means, that the construction of the effective macroscopic properties is only based on $K^{\epsilon}$, but all other data functions are ignored in the homogenization process. Since it is a priorly not clear that this leads to the correct homogenization setting, we will prove later on, under the assumption of periodicity, that this leads in fact to the correct (upscaled) limit problem.
We summarize the fully continuous HMM.
\begin{definition}[Fully continuous HMM for the original system]
\label{hmm-original-system}
Let $K^0_{\epsilon,\kappa}$ be given by (\ref{effective-hydraulic-conductivity}) and recall $\Phi^0$ from assumption (A3). We call $p_{n}^0$, $p_{w}^0$, $s_{n}^0$ and $s_{w}^0$ the fully continuous HMM approximations {\it in classical formulation} if they solve
\begin{align*}
\Phi^0 \partial_t s^0_{\alpha} - \nabla \cdot \left( K^0_{\epsilon,\kappa} \frac{k_{r,\alpha}(s_{\alpha}^0)}{\mu_{\alpha}} ( \nabla p_{\alpha}^0 - \rho_{\alpha} g ) \right) &= 0 \quad \mbox{in} \enspace \Omega_T, \enspace \mbox{for} \enspace \alpha=w,n,\\
s^0_w + s^0_n = 1 \quad \mbox{and} \quad P_c(s_w^0)&= p_{n}^0 - p_{w}^0.
\end{align*}
These equations are understood in the distributional sense. Furthermore, $p_{n}^0$, $p_{w}^0$, $s_{n}^0$ and $s_{w}^0$ are supposed to have the same boundary and initial conditions as the original (fine scale) two-phase flow system.
\end{definition}
Since questions of existence and uniqueness are hard to answer for the system as stated in Definition \ref{hmm-original-system}, we also formulate a fully continuous HMM for the Kirchhoff-transformed system. Both versions are only equivalent under the assumption of sufficient regularity.
\begin{definition}[Fully continuous HMM in cell problem formulation]
\label{new-fully-continuous-hmm}
Let $K^0_{\epsilon,\kappa}$ be given by (\ref{effective-hydraulic-conductivity}) and $\Phi^0$ as in (A3). We call $(S^c,P^c) \in \Epsilon$ the fully continuous HMM approximation in {\it global pressure formulation} if $S^c(\cdot,0)=S_0$ and if $(S^c,P^c) \in \Epsilon$ solves the following effective two-phase flow system:
\begin{eqnarray}
\label{weak-hmm-form-1}\lefteqn{\int_{0}^T \langle \Phi^0 \partial_t S^c(\cdot,t), \Psi(\cdot,t ) \rangle_{H^{-1}(\Omega),H^1_0(\Omega)} \hspace{2pt} dt}\\
\nonumber&=& - \int_{\Omega_T} K^0_{\epsilon,\kappa} (\lambda_{w}(S^c) \nabla P^c + \nabla \Upsilon(S^c) - \lambda_w(S^c) \rho_w g) \cdot \nabla \Psi  \hspace{2pt} dt  \hspace{2pt} dx\\
\label{weak-hmm-form-2}0 &=& \int_{\Omega_T} K^0_{\epsilon,\kappa} ( \lambda(S^c) \nabla P^c - (\lambda_w(S^c) \rho_w + \lambda_n(S^c) \rho_n) g ) ) \cdot \nabla \Psi \hspace{2pt} dt  \hspace{2pt} dx
\end{eqnarray}
for all $\Psi \in L^2((0,T),H^1_0(\Omega))$. Due to classical results (cf. Chen \cite{Chen:2001} for a general survey and see e.g. also \cite{Alt:DiBenedetto:1985,Arbogast:1992,Antontsev:Kazhikhov:Monakhov:1990,Kroener:Luckhaus:1984}), we can guarantee existence of $(S^c,P^c)$ under assumptions (A1)-(A7) and if assumptions (A8) and (A9) are fulfilled the solution $(S^c,P^c) \in \Epsilon$ is also unique.
\end{definition}
Note that Definition \ref{hmm-original-system} and \ref{new-fully-continuous-hmm} lead to multiscale methods that can operate with very low computational costs. First, in a preprocessing step, it only involves once the computation of the cell basis given by $\{w_{\kappa}^i \in L^2(\Omega,\tilde{H}^1_{\sharp}(Y))| \enspace 1 \le i \le d\}$ (which is of the same cost as for the HMM for standard linear elliptic problems). This can be done with a standard finite element method and it can be done in parallel. After that step, one only has to solve the standard two phase flow system, now with coarse scale coefficients.

An additional advantage is that we do not have to deal with the convergence of numerical methods again; existing results can be used because problem (\ref{weak-hmm-form-1})-(\ref{weak-hmm-form-2}) is a standard two-phase problem. The multiscale features are completely hidden in the computation of $K^0_{\epsilon,\kappa}$. We straightforwardly obtain convergence of numerical approximations to the coarse scale solution $(S^c,P^c)$. The remaining error between $(S^c,P^c)$ and $(S^{\epsilon},P^{\epsilon})$ is a small (nonnumerical) modeling error.

Before we draw our attention to a justification of (\ref{weak-hmm-form-1})-(\ref{weak-hmm-form-2}) by homogenization theory, we comment on the possibilities in constructing the cell basis.

\subsection{Notes on the cell problems}
\label{subsection-cell-problem-formulation}

In the formulation of the method we propose a periodic boundary condition for the local problems given by (\ref{definition-cell-basis-eqn}), i.e. the problems that characterize the cell basis. We note that other choices of boundary conditions are possible, but numerical experiments indicate that periodic conditions are quite flexible and they are typically less affected by resonance errors than e.g. Dirichlet or Neumann boundary conditions (c.f. \cite{Abdulle:2005,Gloria:2006}). Resonance errors are errors that arise from a mismatch between the prescribed boundary condition and the wave length of the fine scale variations. As an effect, the local solutions show strong unnatural oscillations close to the boundaries of cells that can distort the final upscaled approximation. The oscillations should be therefore truncated, which is typically done by using oversampling techniques:

\begin{remark}[Oversampling]
\label{oversampling-technique}In general, the periodic boundary condition is a wrong boundary condition for the localized fine scale equations. Since the correct boundary condition is unknown, an oversampling technique is required. This means that the local problems are solved in larger domains, but only the interior information (with a certain distance to the boundary) is used to compute the averages that are communicated to the coarse scale equation. In our example, we only need to change the definition of the intrinsic permeability $K^0_{\epsilon,\kappa}$ by taking the average over a smaller integral. This means that the coarse scale system (\ref{weak-hmm-form-1})-(\ref{weak-hmm-form-2}) remains formally the same, but the intrinsic permeability $K^0_{\epsilon,\kappa}$ is replaced by a new coefficient $K^0_{\epsilon,\kappa,\kappa_0}$ which is defined by
\begin{align*}
\left(K^0_{\epsilon,\kappa,\kappa_0}\right)_{ij}(x) := \int_{Y_{\kappa_0,\kappa}} \hspace{-26pt}-\hspace{7pt} K^{\epsilon}(x+ \kappa y) (e_i + \nablay w_{\kappa}^i(x,y))\cdot  (e_j + \nablay w_{\kappa}^j(x,y)) \hspace{2pt} dy,
\end{align*}
where $Y_{\kappa_0,\kappa} := \left( - \frac{\kappa_0}{2\kappa}, \frac{\kappa_0}{2\kappa} \right)^d$ and $\kappa \ge \kappa_0 \ge \epsilon$.
\end{remark}

\subsection{Homogenization}
\label{subsection-homogenization}

In this section, we restrict ourselves to a periodic setting that is specified by the assumptions below. This section is to justify the method that we proposed in Definition \ref{hmm-original-system} and \ref{new-fully-continuous-hmm}. More precisely, we relate the HMM stated in Definition \ref{new-fully-continuous-hmm} to classical homogenization theory. The finding is that, under the assumption of a periodic structure, $S^c$ and $P^c$ are exactly the homogenized solutions. To prove this result, we make the following assumptions:

\begin{assumption}[For homogenization theory]
We make the following assumptions:
\begin{enumerate}
\item[(A10)] We have $\Phi^{\epsilon}(x)=\Phi(\frac{x}{\epsilon})$, where $\Phi \in L^{\infty}_{\sharp}(Y)$.
\item[(A11)] In addition to assumption (A2), we have $K^{\epsilon}(x)=K(x,\frac{x}{\epsilon})$, where $K \in [L^{\infty}(\Omega, L^{\infty}_{\sharp}(Y))]^{d \times d}$ and such that there exist positive constants $\alpha$ and $\beta$ so that a.e. in $\Omega \times Y$:
\begin{align*}
\alpha |\xi|^2 \le K(x,y) \xi \cdot \xi \le \beta |\xi|^2 \quad \forall \xi \in \mathbb{R}^d.
\end{align*}
Furthermore, $K(x,\frac{x}{\epsilon})$ is meassurable and fulfills 
\begin{align*}
\underset{\epsilon \rightarrow 0}{\mbox{lim}} \int_{\Omega} K_{ij}(x,\frac{x}{\epsilon})^2 \hspace{2pt} dx = \int_{\Omega} \int_Y K_{ij}(x,y)^2 \hspace{2pt} dy \hspace{2pt} dx \quad  1 \le i,j\le d
\end{align*}
(the last two properties are for instance fulfilled if $K \in [C^0(\overline{\Omega}, L^{\infty}_{\sharp}(Y))]^{d \times d}$.
\item[(A12)] Let the boundary saturation be independent of time, i.e. $\bar{S}(x,t)=\bar{S}(x)$, and $P_c(\bar{S}) \in L^{1}(\Omega)$ and $G_{\alpha}(\bar{S}) \in H^{1}(\Omega)$, $\alpha=w,n$, where we defined
\begin{align*}
G_w(\bar{S}):= - \int_0^{\bar{S}}
\frac{\lambda_n(s)}{\lambda(s)} P_c^{\prime}(s) \hspace{2pt} ds \quad \mbox{and} \quad G_n(\bar{S}):= G_w(\bar{S}) + P_c(\bar{S}).
\end{align*}
\item[(A13)]The inverse Kirchhoff transform $\Upsilon^{-1}$ is $\theta$-H\"older-continuous with $0<\theta\le1$ and constant $C_{\Upsilon^{-1}}$.
\end{enumerate}
\end{assumption}

The last assumption (A13) is fulfilled for most of the
capillary pressure -- saturation relationships $P_c(s)$, e.g. the van Genuchten model (cf. \cite{Bear:1972}). We can now formulate the convergence result.

\begin{theorem}[Convergence in the periodic setting]
\label{homogenization-result}We consider a sequence of solutions $(S^{\epsilon},P^{\epsilon}) \in \Epsilon$ to equations (\ref{weak-form-1})-(\ref{weak-form-2}). Let assumptions (A1)-(A7) and (A10)-(A13) be fulfilled and let $\kappa=k \epsilon$, where $k \in \mathbb{N}_{>0}$ is fixed. In this setting, we can replace the definition of $K^0_{\epsilon,\kappa}$ in (\ref{effective-hydraulic-conductivity}) by
\begin{align}
\label{effective-hydraulic-conductivity-periodic}\left(K^0_{\epsilon,\kappa}\right)_{ij}(x) := \int_Y K(x,\frac{x+ \kappa y}{\epsilon}) (e_i + \nablay w_{\kappa}^i(x,y))\cdot  (e_j + \nablay w_{\kappa}^j(x,y)) \hspace{2pt} dy.
\end{align}
By $(S^c,P^c) \in \Epsilon$ we denote the corresponding HMM approximation given by (\ref{weak-hmm-form-1})-(\ref{weak-hmm-form-2}) and with $S^c(\cdot,0)=S_0$. Under these assumptions, there exists a subsequence $(S^{\epsilon^{\prime}},P^{\epsilon^{\prime}})$ of $(S^{\epsilon},P^{\epsilon})$ such that
\begin{align*}
S^{\epsilon^{\prime}} \rightarrow S^c \quad \mbox{in} \enspace L^2(\Omega_T) \quad \mbox{and} \quad P^{\epsilon^{\prime}} \rightharpoonup P^c \quad \mbox{in} \enspace L^2(\Omega_T) \qquad \mbox{as } \epsilon^{\prime} \to 0.
\end{align*}
\end{theorem}
Hence, Theorem \ref{homogenization-result} shows that fully continuous HMM approximations are nothing but the homogenized solutions to (\ref{weak-form-1})-(\ref{weak-form-2}) in the periodic setting. The proof of this theorem is provided in Section \ref{section-proofs}, using two-scale convergence.

\begin{corollary}[Modeling error]
\label{modeling-error-corollary}By modeling error, we refer to the error that arises from the difference between the homogenized coefficient $K^0$ (i.e. the $G$-limit of $K^{\epsilon}$, recall Definition \ref{def-G-convergence}) and the upscaled HMM coefficient $K^0_{\epsilon,\kappa}$. Using definition (\ref{effective-hydraulic-conductivity-periodic}) and under assumption (A11) (i.e. periodicity) this error is equal to zero.
\end{corollary}

A proof of the corollary is also given in Section \ref{section-proofs}.

\begin{remark}
Under the assumptions of Theorem \ref{homogenization-result} (and as a consequence of Corollary \ref{modeling-error-corollary}) we observe that $K^0_{\epsilon,\kappa}$ is independent of $\epsilon$ and $\kappa$. Hence, we can denote $K^0:=K^0_{\epsilon,\kappa}$. On the other hand, comparing (\ref{effective-hydraulic-conductivity}) and (\ref{effective-hydraulic-conductivity-periodic}) (which differ, since one time we use
$K^{\epsilon}(x+ \kappa y) = K(x+ k \epsilon y,\frac{x+ \kappa y}{\epsilon})$
and one time we use $K(x,\frac{x+\kappa y}{\epsilon})$) we see that for Lipschitz-continuous $K$ the $k \epsilon$ perturbation on the
coarse scale vanishes for $\epsilon \rightarrow 0$ in (\ref{weak-hmm-form-1})-(\ref{weak-hmm-form-2}). In this case, our result remains also valid for the original HMM scheme with $K^0_{\epsilon,\kappa}$ given by (\ref{effective-hydraulic-conductivity}), in the sense that $S^{\epsilon^{\prime}} - S^c \rightarrow  0$ in $L^2(\Omega_T)$ (note that $S^c$ depends on $\epsilon$ in this case).
\end{remark}

\section{Numerical treatment of the HMM two-scale problem}
\label{section-numerical-treatment}

In this section, we state two examples of a fully discrete heterogeneous multiscale method (in cell problem formulation) based on the finite elements discretization on the fine scale and a vertex-centered finite volume discretization on the coarse scale. One realization is based on the Kirchhoff transformed equations (\ref{weak-hmm-form-1})-(\ref{weak-hmm-form-2}) and the other one is based on the HMM realization for not-transformed system as stated in Definition \ref{hmm-original-system}. Furthermore, we present an a posteriori estimate for the error between homogenized solutions and HMM approximations. The estimate is up to a modeling error in the sense of Corollary \ref{modeling-error-corollary} and hence equal to zero for the periodic case.

\subsection{Fully discrete HMM approximation}

We introduce the following notations: let $0=t^0<t^1<...<t^N=T$ denote a partition of the time interval $(0,T]$,
with $I_n:=(t^{n-1},t^n]$, $\triangle t^{n}:=t^n-t^{n-1}$ and $\tau
:=$max$\{\triangle t^{n}|1\le n\le N\}$. For each time step $t^n$ ($0\le n \le
N$), let $\mathcal{T}_H^n(\Omega)$ denote
a regular simplicial partition of $\Omega$. The corresponding space of piecewise linear functions is given by:
\begin{align*}
V_H^n(\Omega) := \{ \Psi_H \in C^0(\bar{\Omega})| \hspace{2pt} (\Psi_H)_{|\mathcal{T}}\in\mathbb{P}^1(\mathcal{T}) \enspace \forall \mathcal{T} \in \mathcal{T}_H^n(\Omega)\},
\end{align*}
where $\mathbb{P}^1(\mathcal{T})$ denotes the space of affine functions on $\mathcal{T}$. By $H$ we denote the maximum diameter of an element of $\Tau_H^n(\Omega)$ and $\mathcal{N}_H^n$ defines the set of nodes in $\mathcal{T}_H^n(\Omega)$. For each of the meshes $\mathcal{T}_H^n$, we let $\mathcal{D}_H^n$ denote the corresponding dual tessellation with control volumes $\mathcal{D}$. The dual tessellation $\mathcal{D}_H^n$ is defined as follows: Let $x_i \in \mathcal{N}_H^n$ be a given node, then the corresponding dual volume $\mathcal{D}_i$ is the star-shaped domain with center $x_i$, a polyhedral boundary and where the corners are the points of the set
\begin{align*}
C_i:=\{ \enspace c \in \Omega| \enspace & \mbox{there exists a codim-}k \enspace \mbox{element } E \mbox{ of } \Tau_H^n(\Omega) \mbox{ with }  0 \le k < d\\
&\mbox{and s.t. } c \enspace \mbox{is the barycenter of } E \mbox{ and} \enspace x_i\in\bar{E}\}.
\end{align*}
The union of all these control volumes $\mathcal{D}_i$ defines the dual grid $\mathcal{D}_H^n$. In $2d$, $\mathcal{D}_i$ is obtained by connecting the barycenters of all the elements and edges that are adjacent to $x_i$. 
For $\mathcal{D} \in \mathcal{D}_H^n$, we let $x_{\mathcal{D}}$ define the center of $\mathcal{D}$ (i.e. $x_{\mathcal{D}}$ is a node of $\mathcal{T}_H^n(\Omega)$) and $H_{\D}$ the diameter of the element $\D$. Furthermore, by $\mathcal{D}_H^{\mbox{\tiny int},n} \subset \mathcal{D}_H^n$ we define the set that contains the elements of the dual mesh that belong to the interior nodes $\mathcal{T}_H^n(\Omega)$. Finally, we also define
\begin{align*}
V_{H,\tau}(\Omega_T):= \{ \Psi_{H,\tau} \in C^0(\bar{\Omega_T}) | \enspace &\Psi_{H,\tau}(x,\cdot)_{|I_n} \in \mathbb{P}^1(I_n) \enspace \forall x \in \bar{\Omega}, \enspace 1 \le n \le N;\\
&\Psi_{H,\tau}(\cdot,t^n) \in V_H^n(\Omega), \enspace 0 \le n \le N \}.
\end{align*}
The following assumption only yields a small simplification for formulating the method and deriving a corresponding a posteriori error estimate. We assume that $\bar{S}$ and $\bar{P}$ are piecewise linear in order to avoid boundary approximation in the final scheme.
\begin{assumption}[For formulating the fully discrete HMM]
We make the following assumption:
\begin{enumerate}
\item[(A14)] the boundary functions $\bar{S}$ and $\bar{P}$ are continuous and piecewise linear, i.e. $\bar{S}$, $\bar{P} \in V_{H,\tau}(\Omega_T)$.
\end{enumerate}
\end{assumption}
We can finally define the discrete spaces incorporating the boundary conditions:
\begin{align*}
V_{H;\bar{S}}^n(\Omega) &:= \{ \Psi_H \in V_{H}^n(\Omega)| \hspace{2pt} \Psi_H=\bar{S}(\cdot,t^n) \enspace \mbox{on} \enspace \partial \Omega\},\\
V_{H,\tau;\bar{S}}(\Omega_T)&:=\{ \Psi_H \in V_{H,\tau}(\Omega_T)| \hspace{2pt} \Psi_H=\bar{S}  \enspace \mbox{on} \enspace \partial \Omega \times (0,T] \},\\
V_{H;\bar{P}}^n(\Omega) &:= \{ \Psi_H \in V_{H}^n(\Omega)| \hspace{2pt} \Psi_H=\bar{P}(\cdot,t^n) \enspace \mbox{on} \enspace \partial \Omega\},\\ 
V_{H,\tau;\bar{P}}(\Omega_T) &:= \{ \Psi_H \in V_{H,\tau}(\Omega_T)| \hspace{2pt} \Psi_H=\bar{P}  \enspace \mbox{on} \enspace \partial \Omega \times (0,T] \}.
\end{align*}
In the following, we use the notation $\Psi_{H \tau}^n:=\Psi_{H \tau}(\cdot,t^n)$ for $\Psi_{H \tau} \in V_{H,\tau}(\Omega_T)$ and by $S_{H\tau}^0$ we denote a suitable approximation of the initial value $S_0$.
In order to solve local problems (\ref{definition-cell-basis-eqn}) numerically, we require a regular periodic triangulation of $Y=\left(-\frac{1}{2},\frac{1}{2}\right)^d$ that is denoted by $\mathcal{T}_h(Y)$. We extend $\Tau_h(Y)$ periodically to a triangulation $\Tau_h(\R^d)$ of whole $\R^d$. A corresponding discrete space is given by
\begin{align}
\label{def-W-h} W_h(Y) := \{ \psi_h \in C^0(\bar{Y}) \cap \tilde{H}^1_{\sharp}(Y)| \hspace{2pt} (\psi_h)_{|\mathcal{Y}}\in\mathbb{P}^1(\mathcal{Y}) \enspace \forall \mathcal{Y} \in \mathcal{T}_h(Y)\}
\end{align}
and the set of faces is given by
\begin{align*}
\Gamma (\mathcal{T}_h(Y)) := \{ E | \hspace{2pt} E \hspace{-1pt}=\hspace{-1pt} \mathcal{Y}_1\hspace{-1pt} \cap \hspace{-1pt}\mathcal{Y}_2 \hspace{-1pt}\subset \hspace{-1pt}\Big[\hspace{-4pt}-\frac{1}{2},\frac{1}{2}\Big)^d, \hspace{2pt} \mathcal{Y}_1,\mathcal{Y}_2 \hspace{-1pt}\in\hspace{-1pt} \mathcal{T}_h(\R^d) \enspace \mbox{and} \enspace \mbox{codim}(E)\hspace{-1pt}=\hspace{-1pt}1\}.
\end{align*}
Note that we might identify $\mathcal{T}_h(Y)$ with a triangulation of the $d$-dimensional torus $\mathbb{T}^d$. With this interpretation, we see that $\Gamma (\mathcal{T}_h(Y))$ only contains inner faces, because the torus does not have a boundary. A jump over a face $E \subset \partial Y$ should be therefore seen as a jump over a face on the torus mesh (after mapping it accordingly).

Boundary faces do not exist. A jump over a face $E \subset \partial Y$ must be therefore seen as a jump over a corresponding opposite face. We also denote $h_E:=$diam$(\mathcal{Y} \cup \tilde{\mathcal{Y}})$ where $E = \mathcal{Y} \cap \tilde{\mathcal{Y}} \in \Gamma(\Tauh(Y))$.

$\\$
For the rest of the paper, we denote for simplification
\begin{align}\label{K-eps-kappa-def} 
K_{\epsilon,\kappa}(x,y):=K^{\epsilon}(x + \kappa y).
\end{align}
We first propose a finite element discretization to assemble the local cell basis that defines the discrete effective intrinsic permeability that is communicated to the coarse scale equation:
\begin{definition}[Discrete effective intrinsic permeability]
Let $\kappa \ge \epsilon>0$ and let $x_{\mathcal{D}} \in \mathcal{N}_H^n$ denote a center of an element $\mathcal{D}$ of the dual grid $\mathcal{D}_H^n$. By $K_{\epsilon,\kappa,h}(x_{\mathcal{D}},\cdot)$ we denote an arbitrary piecewise constant approximation of $K_{\epsilon,\kappa}(x_{\mathcal{D}},\cdot)$ (see (\ref{K-eps-kappa-def})) i.e. we assume for all $\D \in \D_H^n$
\begin{align*}
K_{\epsilon,\kappa,h}(x_{\mathcal{D}},\cdot)_{|\mathcal{Y}} \in \mathbb{P}^0(\mathcal{Y}) \quad \forall \mathcal{Y} \in \Tauh(Y).
\end{align*}
For instance, we might use the local mean value on every cell $\mathcal{Y}$.
Then, we denote $w_{\kappa,h}^i(x_{\mathcal{D}},\cdot) \in W_h(Y)$ the solution of
\begin{align*}
\int_Y K_{\epsilon,\kappa,h}(x_{\D},y) (e_i + \nablay w_{\kappa,h}^i(x_{\mathcal{D}},y)) \cdot \nablay \psi_h(y) \hspace{2pt} dy = 0 \quad \forall \psi_h \in W_h(Y).
\end{align*}
The entries of the (piecewise constant) discrete effective intrinsic permeability $K^0_{\epsilon,\kappa,h}$ are now given by
\begin{align*}
\left(\left(K^0_{\epsilon,\kappa,h}\right)_{ij}\right)_{|\mathcal{D}} := \int_{Y} \hspace{-2pt} K_{\epsilon,\kappa,h}(x_{\D},y) (e_i + \nablay w_{\kappa,h}^i(x_{\mathcal{D}},y))\cdot  (e_j + \nablay w_{\kappa,h}^j(x_{\mathcal{D}},y)) \hspace{2pt} dy.
\end{align*}
\end{definition}
Note that the above definition requires a recomputation of the cell basis $(w_{\kappa,h}^i(x_{\mathcal{D}},\cdot))_i$ for each time step if the coarse grid $\Tau_H^n(\Omega)$ changes in time or if it is adaptively refined. Alternatively, in an offline phase, we might compute $w_{\kappa,h}^i(x,\cdot) \in W_h(Y)$ for a large set of $x$-samples $\mathcal{S}\subset \Omega$ that can be reused for any time step. Defining
\begin{align*}
\left(K^0_{\epsilon,\kappa,h}\right)_{ij}(x) := \int_{Y} K_{\epsilon,\kappa,h}(x,y) (e_i + \nablay w_{\kappa,h}^i(x,y))\cdot  (e_j + \nablay w_{\kappa,h}^j(x,y)) \hspace{2pt} dy
\end{align*}
for all $x \in \mathcal{S}$, we may also use a continuous interpolation of the sample values.

$\\$
With the previously defined dual meshes $\mathcal{D}_H^n$, we may use a vertex centered finite volume discretization on the coarse scale (c.f. \cite{Cances:Pop:Vohralik:2012}). The first version of a fully discrete multiscale method is based on the Kirchhoff transformed equation given by Definition \ref{new-fully-continuous-hmm}. 
\begin{definition}[Fully discrete HMM for the Kirchhoff transformed system]
\label{def-fd-hmm-kirchhoff}
Let assumption (A14) be fulfilled. We call $(S_{H\tau},P_{H\tau}) \in V_{H,\tau;\bar{S}}(\Omega_T) \times V_{H,\tau;\bar{P}}(\Omega_T)$ a fully discrete HMM approximation for the Kirchhoff transformed system if there holds
\begin{eqnarray}
\label{fd-hmm-kirchhoff-1}\lefteqn{\int_{\D} \Phi^0 \left( \frac{S_{H\tau}^n - S_{H\tau}^{n-1}}{\triangle t^n} \right) \hspace{2pt} dx}\\
\nonumber&-& \int_{\partial \D} K^0_{\epsilon,\kappa,h}\left( (\lambda_{w}(S_{H\tau}^n) \nabla P_{H\tau}^n + \nabla \Upsilon(S_{H\tau}^n) - \lambda_w(S_{H\tau}^n) \rho_w g) \cdot \nu_{\D}
\right) \hspace{2pt} d \sigma(x) = 0
\end{eqnarray}
and
\begin{align}
\label{fd-hmm-kirchhoff-2}- \int_{\partial \D} K^0_{\epsilon,\kappa,h} ( \lambda(S_{H\tau}^n) \nabla P_{H\tau}^n - (\lambda_w(S_{H\tau}^n) \rho_w + \lambda_n(S_{H\tau}^n) \rho_n) g ) \cdot \nu_{\D}
\hspace{2pt} d \sigma(x) = 0
\end{align}
for all $1 \le n \le N$ and all $\D \in\mathcal{D}_H^{\mbox{\tiny int},n} \subset \mathcal{D}_H^n$. Here, $P_{H\tau}^0$ is the solution of (\ref{fd-hmm-kirchhoff-2}) for given $S_{H\tau}^0$ (which is an approximation of the initial value $S_0$).
\end{definition}
Equations (\ref{fd-hmm-kirchhoff-1}), (\ref{fd-hmm-kirchhoff-2}) form a nonlinear system that might be solved using Newton's method or a fixed point linearization as proposed e.g. in \cite{Cances:Pop:Vohralik:2012}.
Alternatively to the formulation proposed in Definition \ref{def-fd-hmm-kirchhoff}, we might also construct a scheme with unknown $\Upsilon_{H \tau}$ which is an approximation of the Kirchhoff transformed saturation  $\Upsilon(S^c)$. In this case, we only need to replace $S_{H \tau}$ by $\Upsilon^{-1}(\Upsilon_{H \tau})$ in (\ref{fd-hmm-kirchhoff-1})-(\ref{fd-hmm-kirchhoff-2}). The advantage is that we seek an approximation of the unknown $\Upsilon(S^c)$ which has typically more regularity then the non-wetting saturation $S^c$. Again, we refer to \cite{Cances:Pop:Vohralik:2012} for the formulation of such a scheme.

$\\$
The next method is based on the fully continuous HMM for the original system as stated in Definition \ref{hmm-original-system}. The mass balance is solved for both phases and an upwinding term is used for stabilization (c.f. \cite{Cances:Pop:Vohralik:2012,Huber:Helmig:2000}). 
\begin{definition}[Fully discrete HMM for the original system]
\label{def-fd-hmm-original-system}
Let assumption (A14) be fulfilled. According to (\ref{global-pressure-def}), we define the function $P$ that describes the global pressure relation by
\begin{align}
\label{pressure-relation}P(p,s) := p + \int_0^{s} \frac{\lambda_n(\theta)}{\lambda(\theta)} P_c^{\prime}(\theta) \hspace{2pt} d\theta.
\end{align}
Now, we seek $(s_{H \tau; w},p_{H \tau; w}) \in V_{H,\tau;\bar{S}}(\Omega_T) \times V_{H,\tau}(\Omega_T)$ (approximating the functions $s_{w}^0$ and $p_{w}^0$ from Definition \ref{hmm-original-system}) with the properties $P(p_{H \tau; w},s_{H \tau; w})=\bar{P}$ on $\partial \Omega \times (0,T]$ and $s_{H \tau; w}^0=S_{H\tau}^0$, solving
\begin{align}
\label{fd-hmm-1}&\int_{\D} \Phi^0 \left( \frac{s_{H \tau; w}^{n} - s_{H \tau; w}^{n-1}}{\triangle t^n} \right) \hspace{2pt} dx\\
\nonumber&\qquad = \int_{\partial \D} \left[ K^0_{\epsilon,\kappa,h}  \lambda_{w}(s_{H \tau; w}^{n}) \left( \nabla p_{H \tau; w} - \rho_w g \right) \right]^{\mbox{\rm upw}} \cdot \nu_{\D} \hspace{2pt} d \sigma(x) \quad \mbox{and}\\
\label{fd-hmm-2}-&\int_{\D} \Phi^0 \left( \frac{s_{H \tau; w}^{n} - s_{H \tau; w}^{n-1}}{\triangle t^n} \right) \hspace{2pt} dx\\
\nonumber&\qquad= \int_{\partial \D} \left[ K^0_{\epsilon,\kappa,h}  \lambda_{n}(s_{H \tau; w}^{n}) \left( \nabla p_{H \tau; w} + \nabla P_c(s_{H \tau; w}^{n})  - \rho_n g \right) \right]^{\mbox{\rm upw}} \cdot \nu_{\D} \hspace{2pt} d \sigma(x)
\end{align}
for $1 \le n \le N$ and for all $\D \in \mathcal{D}_H^{\mbox{\tiny int},n}$, 
where $[\cdot]^{\mbox{\rm upw}}$ denotes an upwind choice of the evaluation on the element faces.
The non-wetting saturation and the non-wetting pressure are obtained using
\begin{align*}
s_{H \tau; n} := 1 - s_{H \tau; w} \quad \mbox{and} \quad p_{H \tau; n}:=p_{H \tau; w}+P_c(s_{H \tau; w}).
\end{align*}
\end{definition}

\subsection{A posteriori error estimates}

In this section, we present an a posteriori estimate for the error between an arbitrary fully discrete HMM approximation and a homogenized solution that corresponds with problem (\ref{weak-form-1})-(\ref{weak-form-2}).

As in \cite{Cances:Pop:Vohralik:2012}, we start with defining a {\it nonwetting phase flux reconstruction} $u_{s,H\tau}$ and a {\it total flux reconstruction} $u_{p,H\tau}$. We use the subindex $s$ for the flux of the saturation equation and the subindex $p$ for the flux of the pressure equation. The next assumption guarantees existence of $u_{s,H\tau}$ and $u_{p,H\tau}$.
\begin{assumption}[For a-posteriori error estimation]
We make the following assumption:
\begin{enumerate}
\item[(A15)] there exist locally conservative flux reconstructions, i.e. there exist two vector fields $u_{s,H\tau}$ and $u_{p,H\tau}$ that are piecewise constant in time and such that
\begin{align*}
u_{s,H}^n:= (u_{s,H\tau}){|I_n}, \enspace u_{p,H}^n:= (u_{p,H\tau}){|I_n} \in \mathbf{H}(\mbox{\rm div},\Omega) \quad \forall 1 \le n \le N
\end{align*}
and
\begin{align*}
\int_{\D} \frac{S_{H\tau}^n - S_{H\tau}^{n-1}}{\triangle t^n} + \nabla \cdot u_{s,H}^n &= 0 \qquad \forall \D \in \mathcal{D}_H^{\mbox{\tiny int},n}, \enspace \forall n \in \{1,...,N\},\\
\int_{\D} \nabla \cdot u_{p,H}^n &= 0 \qquad \forall \D \in \mathcal{D}_H^{\mbox{\tiny int},n}, \enspace \forall n \in \{1,...,N\}.
\end{align*}
\end{enumerate}
\end{assumption}
Even though the existence of $u_{p,H}^n$ is clear in any case ($u_{s,H}^n=0$ obviously fulfills the property), $0$ is not the flux reconstruction that we are interested in. We want $u_{s,H}^n$ and $u_{p,H}^n$ to be discrete approximations of the exact fluxes that occur in the fully continuous HMM, i.e. $u_{s,H\tau}$ and $u_{p,H\tau}$ are to approximate $u_{s}^0$ and $u_{p}^0$ with
\begin{align*}
u_{s}^0&:= - K^0_{\epsilon,\kappa} (\lambda_{w}(S^c) \nabla P^c + \nabla \Upsilon(S^c) - \lambda_w(S^c) \rho_w g) \enspace \mbox{and}\\
u_{p}^0&:= - K^0_{\epsilon,\kappa} ( \lambda(S^c) \nabla P^c - (\lambda_w(S^c) \rho_w + \lambda_n(S^c) \rho_n) g ) ).
\end{align*}

\begin{remark}[Reconstruction of the fluxes]
A procedure to obtain flux reconstructions for a system of type (\ref{fd-hmm-kirchhoff-1})-(\ref{fd-hmm-kirchhoff-2}) as well as for a system of type (\ref{fd-hmm-1})-(\ref{fd-hmm-2}) in a suitable Raviart-Thomas-N\'{e}d\'{e}lec space is described in \cite[Section 4.2.2 and 4.3.2]{Cances:Pop:Vohralik:2012}. In particular, assumption (A15) is fulfilled for both reconstruction schemes.
\end{remark}

\begin{definition}[Error indicators]
\label{def-error-indicators}
In the following estimators, $\alpha_{\mathcal{D},\kappa}$ denotes the smallest and $\beta_{\mathcal{D},\kappa}$ the largest eigenvalue of $K^{\epsilon}(x_{\D}+ \kappa y)$ in $Y$ (both values can be e.g. determined by solving corresponding eigenvalue cell problems in $\tilde{H}^1_{\sharp}(Y))$. $\alpha$ is the smallest eigenvalue of $K^{\epsilon}$. By $C_P(\D)$ we denote constants that arise from a Poincar\'{e} inequality. For more details on these constants we refer to 
the proof of Theorem \ref{a-posteriori-main-result}.
Let us define the flow functions $V_{H\tau}^s(t)$ and $V_{H\tau}^p(t)$ by
\begin{align*}
V_{H\tau}^s(t)&:=(\lambda_{w}(S_{H\tau}(\cdot,t)) \nabla P_{H\tau}(\cdot,t) + \nabla \Upsilon(S_{H\tau}(\cdot,t)) - \lambda_w(S_{H\tau}(\cdot,t)) \rho_w g),\\
V_{H\tau}^p(t)&:=  \lambda(S_{H\tau}(\cdot,t)) \nabla P_{H\tau}(\cdot,t) - (\lambda_w(S_{H\tau}(\cdot,t)) \rho_w + \lambda_n(S_{H\tau}(\cdot,t)) \rho_n) g.
\end{align*}
Using these flow functions, we define the {\it coarse scale residual estimators} by
\begin{align*}
\eta_{CR,s,\mathcal{D}}^n&:= C_P(\D) H_{\D} \alpha^{-\frac{1}{2}} \|\partial_t S_{H\tau} + \nabla \cdot u_{s,H}^n\|_{L^2(\D)},\\
\eta_{CR,p,\mathcal{D}}^n&:= C_P(\D) H_{\D} \alpha^{-\frac{1}{2}} \| \nabla \cdot u_{p,H}^n\|_{L^2(\D)},
\end{align*}
and the {\it fine scale residual estimators} by
\begin{align*}
\eta_{CF,s,\mathcal{D}}^n(t)&:= \alpha^{-\frac{1}{2}} \frac{\beta_{\mathcal{D},\kappa}}{\alpha_{\mathcal{D},\kappa}} \|V_{H\tau}^s(t)\|_{L^2(\D)} m(\D,\epsilon,\kappa,h),\\
\eta_{CF,p,\mathcal{D}}^n(t)&:= \alpha^{-\frac{1}{2}} \frac{\beta_{\mathcal{D},\kappa}}{\alpha_{\mathcal{D},\kappa}} \|V_{H\tau}^p(t)\|_{L^2(\D)} m(\D,\epsilon,\kappa,h),
\end{align*}
where
\begin{align*}
m(\D,\epsilon,\kappa,h):= \left( \sum_{E \in \Gamma(\Tauh(Y))} h_E \| [K_{\epsilon,\kappa,h}(x_{\D},\cdot) (e_i + \nablay w_{\kappa,h}^i(x_{\D},\cdot))]_{E} \|_{L^2(E)}^2 \right)^{\frac{1}{2}}.
\end{align*}
The {\it diffusive flux estimators} are given by
\begin{align*}
\eta_{DF,s,\mathcal{D}}^n(t)&:= \alpha^{-\frac{1}{2}} \|K^0_{\epsilon,\kappa,h} V_{H\tau}^s(t) + u_{s,H}^n\|_{L^{2} (\D)},\\
\eta_{DF,p,\mathcal{D}}^n(t)&:= \alpha^{-\frac{1}{2}} \|K^0_{\epsilon,\kappa,h} V_{H\tau}^p(t) + u_{s,H}^n\|_{L^{2} (\D)}
\end{align*}
and the {\it approximation error estimators} by
\begin{align*}
&\eta_{APP,s,\mathcal{D}}^n(t) :=\alpha^{-\frac{1}{2}} \frac{\beta_{\mathcal{D},\kappa}}{\alpha_{\mathcal{D},\kappa}} \|V_{H\tau}^s(t)\|_{L^2(\D)} \\
&\hspace{70pt} \cdot \sup_{x \in \D} \|(K_{\epsilon,\kappa}(x,\cdot)-K_{\epsilon,\kappa,h}(x_{\D},\cdot)) (e_i + \nablay w_{\kappa,h}^i(x_{\D},\cdot))\|_{L^2(Y)} \\
&\eta_{APP,p,\mathcal{D}}^n(t):= :=\alpha^{-\frac{1}{2}} \frac{\beta_{\mathcal{D},\kappa}}{\alpha_{\mathcal{D},\kappa}} \|V_{H\tau}^p(t)\|_{L^2(\D)} \\
&\hspace{70pt} \cdot \sup_{x \in \D} \|(K_{\epsilon,\kappa}(x,\cdot)-K_{\epsilon,\kappa,h}(x_{\D},\cdot)) (e_i + \nablay w_{\kappa,h}^i(x_{\D},\cdot))\|_{L^2(Y)}.
\end{align*}
\end{definition}
Let $K^0$ denote some homogenized matrix that we specify in the Theorem \ref{a-posteriori-main-result} and Corollary \ref{a-posteriori-main-result-periodic} below. In the following, by $\|\cdot\|_{E(\Omega)}$ and $\|\cdot\|_{E(\Omega_T)}$ we define energy norms on $H^1_0(\Omega)$ and respectively on $L^2((0,T),H^1_0(\Omega))$ by:
\begin{align*}
\|\Psi\|_{E(\Omega)}&:= \left( \int_{\Omega} K^0(x) \nabla \Psi(x) \cdot \nabla \Psi(x) \hspace{2pt} dx \right)^{\frac{1}{2}} \quad \mbox{and}\\
\|\Psi\|_{E(\Omega_T)}&:= \left( \int_{0}^T \int_{\Omega} K^0(x) \nabla \Psi(x,t) \cdot \nabla \Psi(x,t) \hspace{2pt} dx \hspace{2pt} dt \right)^{\frac{1}{2}}.
\end{align*}
Furthermore, for corresponding functionals $F$, we denote the induced norms on the associated dual spaces by:
\begin{align*}
||| F |||_{\Omega}&:= \underset{\Psi \in H^1_0(\Omega) \setminus\{0\}}{\mbox{\rm sup}}\frac{|F(\Psi)|}{\|\Psi\|_{E(\Omega)}} \quad \mbox{and}\\
||| F |||_{\Omega_T}&:= \underset{\Psi \in L^2((0,T),H^1_0(\Omega)) \setminus\{0\}}{\mbox{\rm sup}}\frac{|F(\Psi)|}{\|\Psi\|_{E(\Omega_T)}}.
\end{align*}

\begin{assumption}[General homogenization setting]
We make the following assumption:
\begin{enumerate}
\item[(A16)] $K^{\epsilon}$ is $G$-convergent to $K^0$, $\Phi^{\epsilon}$ is weak-$\ast$ convergent to $\Phi^0\in(0,1]$ in $L^{\infty}(\Omega)$, $S^{\epsilon}$ converges to $S^0$ strongly in $L^2(\Omega_T)$, $P^{\epsilon}$ converges to $P^0$ weakly in $L^2(\Omega_T)$ and $(S^0,P^0) \in \Epsilon$ with $S^0(\cdot,0)=S_0$ is the unique tuple which fulfills for all $\Psi \in L^2((0,T),H^1_0(\Omega))$
\begin{eqnarray}
\label{weak-hom-form-1}\lefteqn{\int_{0}^T \langle \Phi^0 \partial_t S^0(\cdot,t), \Psi(\cdot,t ) \rangle_{H^{-1}(\Omega),H^1_0(\Omega)} \hspace{2pt} dt}\\
\nonumber&=& - \int_{\Omega_T} K^0 (\lambda_{w}(S^0) \nabla P^0 + \nabla \Upsilon(S^0) - \lambda_w(S^0) \rho_w g) \cdot \nabla \Psi  \hspace{2pt} dt  \hspace{2pt} dx\\
\label{weak-hom-form-2}0 &=& \int_{\Omega_T} K^0 ( \lambda(S^0) \nabla P^0 - (\lambda_w(S^0) \rho_w + \lambda_n(S^0) \rho_n) g ) ) \cdot \nabla \Psi \hspace{2pt} dt  \hspace{2pt} dx.
\end{eqnarray}
\end{enumerate}
\end{assumption}

Before stating the final estimate, we require an additional indicator for the modeling error, i.e. the contribution that describes the error between the real homogenized matrix $K^0$ and the used effective intrinsic permeability $K^0_{\epsilon,\kappa}$ from the fully continuous HMM formulation.
\begin{definition}[Modeling error estimators]
Let $V_{H\tau}^s$ and $V_{H\tau}^p$ be given as in Definition \ref{def-error-indicators}. We define the {\it modeling error estimators} by
\begin{align*}
\eta_{MOD,s,\mathcal{D}}^n(t)&:=\alpha^{-\frac{1}{2}} \|(K^0- K^0_{\epsilon,\kappa}) V_{H\tau}^s(t)\|_{L^{2}(\D)} \\
\eta_{MOD,p,\mathcal{D}}^n(t)&:=\alpha^{-\frac{1}{2}} \|(K^0- K^0_{\epsilon,\kappa}) V_{H\tau}^p(t)\|_{L^{2}(\D)}.
\end{align*}
Note that the modeling error estimators can only be computed in special cases, since the homogenized matrix $K^0$ is typically unknown. In the case of a periodic structure as in Corollary \ref{a-posteriori-main-result-periodic} the modeling error is equal to zero (see Corollary \ref{modeling-error-corollary}). In other cases we need explicit knowledge about the type of the heterogeneities of $K^{\epsilon}$ in order to estimate the modeling error further.
\end{definition}

We are finally prepared to formulate the final a posteriori error estimate.

\begin{theorem}[A posteriori error estimate in the non-periodic case]
\label{a-posteriori-main-result}
Let $(S_{H\tau},P_{H\tau}) \in V_{H,\tau;\bar{S}}(\Omega_T) \times V_{H,\tau;\bar{P}}(\Omega_T)$ denote an arbitrary approximation of $(S^c,P^c)$ and let assumptions (A1)-(A9) and (A14)-(A16) be fulfilled. 
The energy norms $\|\cdot\|_{E(\Omega)}$ and $\|\cdot\|_{E(\Omega_T)}$ are defined with the homogenized matrix $K^0$. Then there holds the following a posteriori error estimate
\begin{eqnarray}
\nonumber\lefteqn{ ||| S_{H\tau} - S^0 |||^2_{\Omega_T} + \| P_{H\tau} -  P^0 \|_{E(\Omega_T)}^2 + \| \Upsilon(S_{H\tau}) - \Upsilon(S^0) \|_{L^2(\Omega_T)}^2}\\
\label{equation-a-post}&\lesssim& |||S_{H\tau}(\cdot,0) - S_0 |||^2_{\Omega} + \sum_{\alpha=s,p} \sum_{n=1}^N \sum_{\mathcal{D} \in \mathcal{D}_H^{n}} \int_{I_n} \eta_{MOD,\alpha,\mathcal{D}}^n(t)^2 \\
\nonumber&\enspace&+
\sum_{\alpha=s,p} \sum_{n=1}^N \sum_{\mathcal{D} \in \mathcal{D}_H^{n}} \int_{I_n} \left( \eta_{CR,\alpha,\mathcal{D}}^n + (\eta_{CF,\alpha,\mathcal{D}}^n + \eta_{DF,\alpha,\mathcal{D}}^n + \eta_{APP,\alpha,\mathcal{D}}^n)(t) \right)^2.
\end{eqnarray}
\end{theorem}
Theorem \ref{a-posteriori-main-result} yields (localized) error indicators for each of the error contributions. We have indicators $\eta_{CR,\alpha,\mathcal{D}}^n$ for the coarse scale residual that depend on the diffusive flux reconstructions. The accuracy of the flux reconstruction itself is evaluated using the indicators $\eta_{DF,\alpha,\mathcal{D}}^n$. Residual error indicators for the solutions of the cell problems are given by $\eta_{CF,\alpha,\mathcal{D}}^n$. Finally, the error for replacing $K^{\epsilon}$ by a piecewise constant approximation is evaluated by the indicators given by $\eta_{APP,\alpha,\mathcal{D}}^n$. However, the modeling error part $\eta_{MOD,\alpha,\mathcal{D}}^n$ is typically unknown.

If we want to apply one of the proposed heterogeneous multiscale methods to a homogenization problem with unknown micro structure (i.e. unknown modeling error), we can still use (\ref{equation-a-post}) to evaluate the discretization error and to construct adaptive mesh refinement strategies. The (possibly not computable) modeling error contribution is a typical remainder that is due to the chosen method and cannot be reduced by changing the discretization (c.f. \cite{Abdulle:Nonnenmacher:2011,Henning:Ohlberger:2011_2}). It is a bounded term that typically converges to zero for $\epsilon$ converging to zero independently of $H$ and $\tau$ (c.f. \cite{Gloria:2006,Gloria:2008}). In particular, if we replace $S^0$ and $P^0$ by $S^c$ and $P^c$ on the left side of (\ref{equation-a-post}), the modeling error contributions on the right side vanish and we get a fully computable a-posteriori error estimator for the discretization error.

For the case of a periodic we do not encounter a problem with the modeling error. Using Corollary \ref{modeling-error-corollary} we can formulate the following improved error estimate.

\begin{corollary}[A-posteriori error estimate in the periodic case]
\label{a-posteriori-main-result-periodic} 
Let assumptions (A1)-(A14) be fulfilled. In particular, we are in the case of periodic coefficient functions. Let $K^0$ denote the homogenized matrix stated in Theorem \ref{homogenization-result} (i.e. $K^0$ denotes the $G$-limit of $K^{\epsilon}$), $P^0$ the weak $L^2(\Omega_T)$-limit of $P^{\epsilon}$, $S^0$ the strong $L^2(\Omega_T)$-limit of $S^{\epsilon}$ and $\Phi^0=1$ (otherwise we rescale $K^0$ by dividing it by $\Phi^0>0$).
Then, for an arbitrary approximation $(S_{H\tau},P_{H\tau}) \in V_{H,\tau;\bar{S}}(\Omega_T) \times V_{H,\tau;\bar{P}}(\Omega_T)$ of $(S^c,P^c)$ fulfilling assumption (A15) there holds the following a posteriori error estimate
\begin{eqnarray}
\nonumber\lefteqn{ ||| S_{H\tau} - S^0 |||^2_{\Omega_T} + \| P_{H\tau} -  P^0 \|_{E(\Omega_T)}^2 + \| \Upsilon(S_{H\tau}) - \Upsilon(S^0) \|_{L^2(\Omega_T)}^2}\\
\label{equation-a-post-periodic}&\lesssim& |||S_{H\tau}(\cdot,0) - S_0 |||^2_{\Omega} \\
\nonumber&\enspace&+
\sum_{\alpha=s,p} \sum_{n=1}^N \sum_{\mathcal{D} \in \mathcal{D}_H^{n}} \int_{I_n} \left( \eta_{CR,\alpha,\mathcal{D}}^n + (\eta_{CF,\alpha,\mathcal{D}}^n + \eta_{DF,\alpha,\mathcal{D}}^n + \eta_{APP,\alpha,\mathcal{D}}^n)(t) \right)^2,
\end{eqnarray}
with a fully computable right hand side. Since $\Upsilon^{-1}\in C^{0,\theta}$ by assumption (A13), we can replace $\| \Upsilon(S_{H\tau}) - \Upsilon(S^0) \|_{L^2(\Omega_T)}^2$ by $\| S_{H\tau} - S^0 \|_{L^{1+\theta}(\Omega_T)}^{1+\theta}$.
\end{corollary}

\begin{conclusion}
Let assumptions (A1)-(A14) be fulfilled and let $(S_{H\tau},P_{H\tau}) \in V_{H,\tau;\bar{S}}(\Omega_T) \times V_{H,\tau;\bar{P}}(\Omega_T)$ denote the solution of the fully discrete heterogeneous multiscale method proposed in Definition \ref{def-fd-hmm-kirchhoff} or let $(s_{H \tau; w},p_{H \tau; w}) \in V_{H,\tau;\bar{S}}(\Omega_T) \times V_{H,\tau}(\Omega_T)$ denote the solution of the method stated in Definition \ref{def-fd-hmm-original-system}. Then, $(S_{H\tau},P_{H\tau})$ and respectively $(S_{H\tau},P_{H\tau}):=(s_{H \tau; w},P(p_{H \tau; w},s_{H \tau; w}))$ (with $P$ given by (\ref{pressure-relation})) fulfill the a posteriori error estimate (\ref{equation-a-post-periodic}).
\end{conclusion}

\begin{remark}[Efficiency]
Efficiency of the total estimated error (i.e. the term on the right hand side of (\ref{equation-a-post-periodic})) can be shown using the techniques from \cite{Cances:Pop:Vohralik:2012}, where a corresponding result is derived for a method of the same structure as the scheme in Definition \ref{def-fd-hmm-kirchhoff}. In particular, the sum of the estimators forms a lower bound for the residuals in the dual norm.
\end{remark}

\section{Proofs of the main results}
\label{section-proofs}

In this section we are concerned with the proofs of Theorems \ref{homogenization-result} and \ref{a-posteriori-main-result}. We do not prove Corollary \ref{a-posteriori-main-result-periodic} separately, since it is an easy conclusion from Theorem \ref{homogenization-result} and Theorem \ref{a-posteriori-main-result}.

\subsection{Proof of Theorem \ref{homogenization-result} (Convergence in the periodic setting)}

The first section is devoted to the homogenization of the Kirchhoff transformed two-phase flow equations under the assumption of periodicity.

The strategy is to derive the homogenized problem associated with the original weak problem (\ref{weak-form-1})-(\ref{weak-form-2}) under the assumptions (A1)-(A7) and (A10)-(A13). Then, we verify that the homogenized system is identical to the two-scale HMM given by (\ref{weak-hmm-form-1})-(\ref{weak-hmm-form-2}).

In the following, we make use of the tools derived in \cite{Amaziane:et-al:2010} for immiscible compressible two-phase flow in porous media. In particular, we need the following compactness Lemma obtained in \cite{Amaziane:et-al:2010}:
\begin{lemma}
\label{compactness-lemma-amaziane-et-al}
Let $\Phi \in L^{\infty}_{\sharp}(Y)$ with $\phi^*\le\Phi(y)\le\Phi^*$ for a.e. $y \in Y$ and $\phi^*,\Phi^*\in(0,1)$. Let $(v^{\epsilon})_{\epsilon>0} \subset L^2(\Omega_T)$ fulfill the following properties:
\begin{enumerate}
\item $0 \le v^{\epsilon} \le C$ a.e. in $\Omega_T$ and for all $\epsilon$.
\item There exists a function $\varpi$ with $\varpi(\xi)\rightarrow 0$ for $\xi \rightarrow 0$ such that the following inequality holds uniformly in $\epsilon$:
\begin{align*}
\int_{\Omega_T} |v^{\epsilon}(x+\triangle x,t)-v^{\epsilon}(x,t)|^2 \hspace{2pt} dx \hspace{2pt} dt \le C\varpi(|\triangle x|).
\end{align*}
\item The functions $v^{\epsilon}$ are such that for all $\epsilon$:
\begin{align*}
\left\| \Phi(\frac{\cdot}{\epsilon}) \partial_t v^{\epsilon} \right\|_{L^2((0,T),H^{-1}(\Omega))} \le C.
\end{align*}
\end{enumerate}
Then $(v^{\epsilon})_{\epsilon>0}$ is a precompact set in $L^2(\Omega_T)$.
\end{lemma}
Again, following the analysis presented in \cite{Amaziane:et-al:2010}  for compressible two phase flow, the following estimates hold true and will allow to extract convergent subsequences of $P^{\epsilon}$, $\Upsilon^{\epsilon}$ and $S^{\epsilon}$:
\begin{lemma}
\label{a-priori-estimates-for-homogenization}
Let assumptions (A1)-(A13) be fulfilled and let $(S^{\epsilon},P^{\epsilon}) \in \Epsilon$ with $S^{\epsilon}(\cdot,0)=S_0$ denote the sequence of solutions of the Kirchhoff transformed two-phase flow system given by 
(\ref{weak-form-1})-(\ref{weak-form-2}). Then the following a priori estimates hold with an $\epsilon$-independent constants $C$
\begin{align*}
\| \nabla P^{\epsilon} \|_{L^2(\Omega_T)} \le C, \quad \| \nabla \Upsilon^{\epsilon} \|_{L^2(\Omega_T)} &\le C, \quad
\left\| \Phi^{\epsilon} \partial_t S^{\epsilon} \right\|_{L^2((0,T),H^{-1}(\Omega))} \le C,\\
{ \int_{\Omega_T} |S^{\epsilon}(x+\triangle x,t)-S^{\epsilon}(x,t)|^2 \hspace{2pt} dx \hspace{2pt} dt} &{\le C\varpi(|\triangle x|).}
\end{align*}
\end{lemma}

We are now prepared to formulate and prove the first convergence result in two-scale homogenized form. 
\begin{theorem}[Two-scale homogenized system]
\label{two-scale-homogenized-system}Let assumptions (A1)-(A7) and (A10)-(A13) be fulfilled. Then the two-scale homogenized system associated with (\ref{weak-hmm-form-1}) and (\ref{weak-hmm-form-2}) reads: find $(S^0,P^0) \in \Epsilon$ and the correctors ${\Upsilon}^1\in L^2((0,T) \times \Omega,\tilde{H}^1_{\sharp}(Y))$ and ${P}^1 \in L^2((0,T) \times \Omega,\tilde{H}^1_{\sharp}(Y))$ such that
\begin{eqnarray*}
\label{tsh-saturation}\nonumber\lefteqn{\int_{0}^T \langle \Phi^0 \partial_t S^0(\cdot,t), \Psi(\cdot,t ) \rangle_{H^{-1}(\Omega),H^1_0(\Omega)} \hspace{2pt} dt}\\
&=& - \int_{\Omega_T} \int_Y K (\lambda_{w}(S^0) (\nablax P^0 + \nablay P^1)) \cdot (\nablax \Psi + \nablay \psi)\\
\nonumber&\enspace& - \int_{\Omega_T} \int_Y K ((\nablax \Upsilon(S^0) + \nablay {\Upsilon}^1) - \lambda_w(S^0) \rho_w g) \cdot (\nablax \Psi + \nablay \psi)
\end{eqnarray*}
and
\begin{eqnarray}
\label{tsh-pressure}0 = \int_{\Omega_T} \int_Y K ( \lambda(S^0) (\nablax P^0 + \nablay P^1) - (\lambda_w(S^0) \rho_w + \lambda_n(S^0) \rho_n) g ) ) \cdot \left( \nablax \Psi + \nablay \psi \right)
\end{eqnarray}
for all $\Psi \in L^2((0,T),H^1_0(\Omega))$ and $\psi \in L^2(\Omega_T,H^1_{\sharp}(Y))$. 
\end{theorem}
This system above has a unique solution. This is easy to verify as soon as we reformulate the problem into macroscopic formulation (see also end of proof of Theorem \ref{homogenization-result} below).

\begin{proof}[Proof of Theorem \ref{two-scale-homogenized-system}]
In the following, we make use of the concept of two-scale convergence (c.f. \cite{Allaire:1992} or \cite{Lukkassen:Nguetseng:Wall:2002} for an overview on this topic). A sequence $u^{\epsilon} \in L^2(\Omega)$ is called two-scale convergent to a limit $u_0 \in L^2(\Omega \times Y)$ if the following holds for all $\phi \in L^2(\Omega, C^0_{\sharp}(Y))$:
\begin{align*}
\lim_{\epsilon \rightarrow 0} \int_{\Omega} u^{\epsilon}(x) \phi(x,\frac{x}{\epsilon}) \hspace{2pt} dx = \int_{\Omega} \int_{Y} u_0(x,y) \phi(x,y) \hspace{2pt} dy \hspace{2pt} dx.
\end{align*}
As proved by Allaire \cite{Allaire:1992}, bounded sequences in $H^1(\Omega)$ allow to extract two-scale convergent subsequences, i.e. for any bounded set $(u^{\epsilon})_{\epsilon >0} \subset H^1(\Omega)$, there exists a subsequence $(u^{\epsilon^{\prime}})_{\epsilon^{\prime} >0}$ and functions $u_0 \in H^1(\Omega)$ and $u_1 \in L^2(\Omega, \tilde{H}^1_{\sharp}(Y))$ such that $u^{\epsilon^{\prime}} \rightharpoonup u_0$ in $H^1(\Omega)$ and $\nabla u^{\epsilon^{\prime}} \overset{\mbox{2-Sc.}}{\rightharpoonup} \nablax u_0 + \nablay u_1$. Using this well know result, the a priori estimates in Lemma \ref{a-priori-estimates-for-homogenization} and the compactness   in Lemma \ref{compactness-lemma-amaziane-et-al}, we get the following convergence up to a subsequence (still denoted by $\epsilon$): there exist functions $S^0 \in L^2(\Omega_T)$, $\Upsilon^0\in L^2((0,T),H^1(\Omega))$, $\Upsilon^1\in L^2((0,T) \times \Omega,\tilde{H}^1_{\sharp}(Y))$, $P^0 \in L^2((0,T),H^1(\Omega))$ and $P^1 \in L^2((0,T) \times \Omega,\tilde{H}^1_{\sharp}(Y))$ such that
\begin{align*}
S^{\epsilon} \rightarrow S^0 \enspace \mbox{in} \enspace L^2(\Omega_T), \enspace \nabla \Upsilon^{\epsilon} \overset{\mbox{2-Sc.}}{\rightharpoonup} \nablax \Upsilon^0 + \nablay \Upsilon^1 \enspace \mbox{and} \enspace \nabla P^{\epsilon} \overset{\mbox{2-Sc.}}{\rightharpoonup} \nablax P^0 + \nablay P^1,
\end{align*}
where we have $\Upsilon^0= \Upsilon(S^0)$, because of $\Upsilon^{\epsilon}\rightharpoonup\Upsilon^0$ in $L^2((0,T),H^1(\Omega))$ and $\Upsilon^{\epsilon}\rightarrow \Upsilon(S^0)$ in $L^2(\Omega_T)$. In particular, we also have $P^{\epsilon} \rightharpoonup P^0 \quad \mbox{in} \enspace L^2(\Omega_T)$.

Next, we use test functions of the form
\begin{align*}
\Psi^{\epsilon}(x,t) = \Psi(x,t) + \epsilon \psi(x,\frac{x}{\epsilon},t), \quad \mbox{where} \enspace (\Psi,\psi)\in C_c^{\infty}(\Omega_T) \times C_c^{\infty}(\Omega_T,C^{\infty}_{\sharp}(Y)),
\end{align*}
in the weak two-phase flow system given by (\ref{weak-form-1})-(\ref{weak-form-2}). Using the admissibility of $K^{\epsilon}$ (i.e. assumption (A11), c.f. \cite{Allaire:1992}) and the two-scale convergence, we can form the limits with respect to $\epsilon$ to obtain:
\begin{eqnarray*}
\lefteqn{\int_{0}^T \langle \Phi^0 \partial_t S^0(\cdot,t), \Psi(\cdot,t ) \rangle_{H^{-1}(\Omega),H^1_0(\Omega)} \hspace{2pt} dt}\\
&=& - \int_{\Omega_T} \int_Y K(x,y) (\lambda_{w}(S^0(x)) (\nablax P^0(x) + \nablay P^1(x,y))) \cdot (\nablax \Psi(x) + \nablay \psi(x,y)) \hspace{2pt} dy \hspace{2pt} dx\\
&\enspace& - \int_{\Omega_T} \int_Y K(x,y) ((\nablax \Upsilon(S^0(x)) + \nablay \Upsilon^1(x,y)) - \lambda_w(S^0(x)) \rho_w g) \cdot (\nablax \Psi(x) + \nablay \psi(x,y)) \hspace{2pt} dy \hspace{2pt} dx
\end{eqnarray*}
and
\begin{eqnarray*}
\lefteqn{0 = \int_{\Omega_T} \int_Y K(x,y) ( \lambda(S^0(x)) (\nablax P^0(x) + \nablay P^1(x,y)) - (\lambda_w(S^0(x)) \rho_w + \lambda_n(S^0(x)) \rho_n) g )} \\
&\enspace& \hspace{30pt} \cdot \left( \nablax \Psi(x) + \nablay \psi(x,y) \right) \hspace{2pt} dy \hspace{2pt} dx \hspace{200pt}
\end{eqnarray*}
for all $\Psi \in L^2((0,T),H^1_0(\Omega))$ and $\psi \in L^2(\Omega_T,H^1_{\sharp}(Y))$ (due to density of the smooth functions). Boundary and initial values for $S^0$ and $P^0$ remain valid since they hold for the whole $\epsilon$-sequence. Now, choosing first $\Psi=0$ and then $\psi=0$ we can decouple the problem into coarse scale and fine scale equations.
\end{proof}

With the previous result, we can now reformulate the two-scale homogenized system until it has the HMM form (\ref{weak-hmm-form-1})-(\ref{weak-hmm-form-2}).

\begin{proof}[Proof of Theorem \ref{homogenization-result}]
We start from the solution $(S^c,P^c) \in \Epsilon$ of the fully continuous HMM problem (\ref{weak-hmm-form-1})-(\ref{weak-hmm-form-2}) (of which we know that it exists and that it is unique). Let us for simplicity assume that $\kappa= \epsilon$ (the general case with $\kappa= k \epsilon$ for $k \in \mathbb{N}$ follows directly, because we can exploit the periodicity and glue the solutions for $k=1$ together to create the solution for arbitrary $k$). Let $w_{\kappa}^i \in L^2(\Omega,\tilde{H}^1_{\sharp}(Y))$ denote the cell basis elements in the periodic setting for $k=1$, i.e. solving
\begin{align*}
\int_Y K(x,\frac{x+ \epsilon y}{\epsilon}) (e_i + \nablay w_{\epsilon}^i(x,y)) \cdot \nablay \psi(y)  \hspace{2pt} dy = 0 
\end{align*}
for all $\psi \in \tilde{H}^1_{\sharp}(Y)$.  The definition and the usage of periodicity of $K(x,\cdot)$ and $w_{\epsilon}^i(x,\cdot)$ implies that
\begin{align}
\label{def-K-0-nonsymmetric}\left(K^0_{\epsilon,\epsilon}\right)_{ij}(x) 
\nonumber &=\int_Y K(x,\frac{x+ \epsilon y}{\epsilon}) (e_i + \nablay w_{\epsilon}^i(x,y))\cdot e_j \hspace{2pt} dy\\
&=\int_Y K(x,y) (e_i + \nablay w^i(x,y))\cdot e_j \hspace{2pt} dy
\end{align}
with $w_i(x,y):=w_{\epsilon}^i(x,y-\frac{x}{\epsilon})$. We now define $P^f(x,\cdot,t)\in \tilde{H}^1_{\sharp}(Y)$ by
\begin{align*}
P^f := \sum_{i=1}^n w_i \left( \partial_{x_i}P^c - \frac{(\lambda_w(S^c(x,t)) \rho_w + \lambda_n(S^c(x,t)) \rho_n)}{\lambda(S^c(x,t))} g_i \right).
\end{align*}
Using the definition of $w_i$ we get
\begin{eqnarray}
\nonumber\label{result-eqn-1}\lefteqn{\int_Y K(x,y) \lambda(S^c(x,t)) \nablay P^f(x,y,t) \cdot \nablay \psi(y) \hspace{2pt} dy}\\
&=& - \int_Y K(x,y) \lambda(S^c(x,t)) \nablax P^c(x,t) \cdot \nablay \psi(y) \hspace{2pt} dy \\
\nonumber&\enspace& \quad + \int_Y K(x,y) (\lambda_w(S^c(x,t)) \rho_w + \lambda_n(S^c(x,t)) \rho_n) g \cdot \nablay \psi(y) \hspace{2pt} dy.
\end{eqnarray}
Next, we make the simplifying notation $G := (\lambda_w(S^c) \rho_w + \lambda_n(S^c) \rho_n) g$. Plugging this, the definition of $K^0_{\epsilon,\epsilon}$ and the definition of $P^f$ into the HMM pressure equation (\ref{weak-hmm-form-2}) gives for $\Psi \in L^2((0,T),H^1_0(\Omega))$:
\begin{align*}
0 &= \int_{\Omega_T} \lambda(S^c) K^0_{\epsilon,\epsilon} \nablax P^c \cdot \nabla \Psi - \int_{\Omega_T} K^0_{\epsilon,\kappa} G \cdot \nablax \Psi \\
&\overset{(\ref{def-K-0-nonsymmetric})}{=} \int_{\Omega_T} \lambda(S^c) K^0_{\epsilon,\epsilon} \nablax P^c \cdot \nabla \Psi  - \int_{\Omega_T} \left( \int_Y K \left(  G + \sum_{i=1}^n (w_i G_i) \right) \right) \cdot \nablax \Psi \\
&= \int_{\Omega_T} \int_Y K \left( \lambda(S^c) (\nablax P^c + \sum_{i=1}^n w_i \partial_{x_i}P^c)
- \sum_{i=1}^n (w_i G_i) - G \right) \cdot \nablax \Psi \\
&= \int_{\Omega_T} \int_Y K ( \lambda(S^c) (\nablax P^c + \nablay P^f) - G ) ) \cdot \nablax \Psi.
\end{align*}
Together with (\ref{result-eqn-1}), we observe that $(P^c,P^f)$ just solves (\ref{tsh-pressure}), which verifies the pressure equation.

Next, we deal with the saturation. For given $P^c$, $P^f$ and $S^c$, let $\Upsilon^f(x,\cdot,t) \in \tilde{H}^1_{\sharp}(Y)$ denote the solution of
\begin{eqnarray}
\nonumber\label{result-eqn-2}0 = \int_Y K (\lambda_{w}(S^c) (\nablax P^c + \nablay P^f) + (\nablax \Upsilon(S^c) + \nablay \Upsilon^f) - \lambda_w(S^c) \rho_w g) \cdot \nablay \psi 
\end{eqnarray} 
for all $\psi \in \tilde{H}^1_{\sharp}(Y)$. Furthermore, we define $\tilde{ \Upsilon}^f:=(\lambda_{w}(S^c) P^f + \Upsilon^f)$ which solves
\begin{eqnarray*}
\int_Y K ( V + \nablay \tilde{ \Upsilon}^f) \cdot \nablay \psi  = 0 \quad \forall \psi \in \tilde{H}^1_{\sharp}(Y)
\end{eqnarray*} 
where $V := \lambda_{w}(S^c) \nablax P^c + \nablax \Upsilon(S^c) - \lambda_w(S^c) \rho_w g$.
We immediately verify the relation $\tilde{ \Upsilon}^f = \sum_{i=1}^n w_i V_i$. Inserting this in the macro equation for the saturation (\ref{tsh-saturation}) gives us for $\Psi \in L^2((0,T),H^1_0(\Omega))$:
\begin{eqnarray*}
\lefteqn{\int_{0}^T \langle \Phi^0 \partial_t S^c(\cdot,t), \Psi(\cdot,t ) \rangle_{H^{-1}(\Omega),H^1_0(\Omega)} \hspace{2pt} dt}\\
&=& - \int_{\Omega_T} K^0_{\epsilon,\epsilon} (\lambda_{w}(S^c) \nablax P^c + \nablax \Upsilon(S^c) - \lambda_w(S^c) \rho_w g) \cdot \nablax \Psi \\
&=& - \int_{\Omega_T} K^0_{\epsilon,\kappa} V \cdot \nablax \Psi
\overset{(\ref{def-K-0-nonsymmetric})}{=} - \int_{\Omega_T}\int_Y K (V  +  \sum_{i=1}^n V_i \nablay w_i) \cdot \nablax \Psi \\
&=& - \int_{\Omega_T}\int_Y K (\lambda_{w}(S^c) \nablax P^c + \nablax \Upsilon(S^c) + \nablay \tilde{ \Upsilon}^f - \lambda_w(S^c) \rho_w g) \cdot \nablax \Psi \\
&=& - \int_{\Omega_T}\int_{Y} K (\lambda_{w}(S^c) \left( \nablax P^c + \nablay P^f\right) \cdot \nablax \Psi \\
&\enspace& \quad - \int_{\Omega_T}\int_{Y} K \left( \nablax \Upsilon(S^c) + \nablay \Upsilon^f - \lambda_{w}(S^c) \rho_w g \right) \cdot \nablax \Psi.
\end{eqnarray*}
Together with (\ref{result-eqn-2}), we see that $(P^c,P^f)$ also solves (\ref{tsh-saturation}). Hence: $(P^c,P^f)=(P^0,P^1)$ (since all equations work in both directions). So the HMM approximation is identical to the homogenized solution (which is just the limit for $\epsilon\rightarrow 0$). The existence of $(P^c,P^f)$ is fulfilled,  because $K^0_{\epsilon,\kappa}$ is a matrix fulfilling assumption (A2). Since the assumptions (A1) and (A3)-(A7) are also still valid we can again use the existence result by Chen (c.f. \cite[Theorem 2.1]{Chen:2001} ) to deduce existence of $(S^c,P^c) \in \Epsilon$. If additionally (A8) and (A9) are also valid this solution is unique (c.f. \cite[Theorem 3.1]{Chen:2001}). By the just demonstrated equivalence, we obtain the same result for $(P^0,P^1)$.
\end{proof}
\begin{remark}[Homogenized problem in cell problem formulation]
Assume that (A1)-(A7) and (A10)-(A13) hold true. From the above proof we conclude that the two-scale homogenized system presented in Theorem \ref{two-scale-homogenized-system} can be also expressed in a cell problem formulation. This leads to an ($\epsilon$-independent) cell basis $\{w^i \in L^2(\Omega,\tilde{H}^1_{\sharp}(Y))| \enspace 1 \le i \le d\}$ with $w^i \in L^2(\Omega,\tilde{H}^1_{\sharp}(Y))$ solving
\begin{align*}
\int_Y K(x,y) (e_i + \nablay w^i(x,y)) \cdot \nablay \psi(y) \hspace{2pt} dy = 0 \quad \psi \in \tilde{H}^1_{\sharp}(Y)
\end{align*}
and an effective conductivity $K^0$ given by
\begin{align}
\label{periodic-case-exact-hom-coeff}\left(K^0\right)_{ij}(x) := \int_Y K(x,y) (e_i + \nablay w^i(x,y))\cdot  (e_j + \nablay w^j(x,y)) \hspace{2pt} dy.
\end{align}
Then, the homogenized solution $(S^0,P^0) \in \Epsilon$ with $S^0(\cdot,0)=S_0$ solves the following effective two phase flow system for all $\Psi \in L^2((0,T),H^1_0(\Omega))$
\begin{eqnarray*}
\lefteqn{\int_{0}^T \langle \Phi^0 \partial_t S^0(\cdot,t), \Psi(\cdot,t ) \rangle_{H^{-1}(\Omega),H^1_0(\Omega)} \hspace{2pt} dt}\\
\nonumber&=& - \int_{\Omega_T} K^0 (\lambda_{w}(S^0) \nabla P^0 + \nabla \Upsilon(S^0) - \lambda_w(S^0) \rho_w g) \cdot \nabla \Psi  \hspace{2pt} dt  \hspace{2pt} dx\\
0 &=& \int_{\Omega_T} K^0 ( \lambda(S^0) \nabla P^0 - (\lambda_w(S^0) \rho_w + \lambda_n(S^0) \rho_n) g ) ) \cdot \nabla \Psi \hspace{2pt} dt \hspace{2pt} dx.
\end{eqnarray*}
\end{remark}

\begin{proof}[Proof of Corollary \ref{modeling-error-corollary}]
From classic homogenization theory (cf. \cite{MuT97}) we know that a coefficient $K^{\epsilon}$ that fulfills (A11) (i.e. which is locally periodic) has the unique $G$-limit $K^0$, with $K^0$ given by (\ref{periodic-case-exact-hom-coeff}). On the other hand, we have seen in the proof of Theorem \ref{homogenization-result} that the coefficient $K^0_{\epsilon,\kappa}$ (given by (\ref{effective-hydraulic-conductivity-periodic})) can be transformed into $K^0$ by a simple integral transformation (see in particular (\ref{def-K-0-nonsymmetric})). Hence $K^0_{\epsilon,\kappa}$ is independent of $\epsilon$ and $\kappa$ and the modeling error is zero.
\end{proof}

\subsection{Proof of  Theorem \ref{a-posteriori-main-result} (A posteriori error estimate)}

In this section we will briefly sketch how to prove the a posteriori error estimate stated in Theorem \ref{a-posteriori-main-result}. To a large extend one can proceed as in the recent work by Canc\`{e}s, Pop and Vohral\'{i}k \cite{Cances:Pop:Vohralik:2012} where a rigorous a posteriori error estimate for the two-phase flow model is derived. Because of this, we just shortly discuss the two main steps and refer to the arguments in \cite{Cances:Pop:Vohralik:2012}. The main difference to the setting in \cite{Cances:Pop:Vohralik:2012} is that we need an additional treatment of the effective intrinsic permeability $K^0_{\epsilon,\kappa}$. A minor difference is the presence of gravity terms, which are ignored in \cite{Cances:Pop:Vohralik:2012}. Both details do not lead to crucial changes and keep the proves rather straightforward.

We start with defining the pressure and saturation residuals. More precisely, for $K^0$ denoting the homogenized matrix as in Theorem \ref{a-posteriori-main-result} and for $(S_{H\tau},P_{H\tau}) \in \Epsilon$, we define the pressure residual $\mathcal{R}_p(S_{H\tau},P_{H\tau}) \in L^2((0,T),H^1_0(\Omega))$ and the saturation residual $\mathcal{R}_s(S_{H\tau},P_{H\tau}) \in L^2((0,T),H^1_0(\Omega))$ by:
\begin{align*}
\langle \mathcal{R}_p(S_{H\tau},P_{H\tau}),\Psi \rangle  &:= \int_{\Omega_T} \lambda(S_{H\tau}) K^0 \nabla P_{H\tau} \cdot \nabla \Psi \\
&\enspace \hspace{-50pt}- \int_{\Omega_T} K^0 (\lambda_w(S_{H\tau}) \rho_w + \lambda_n(S_{H\tau}) \rho_n) g \cdot \nabla \Psi,\\
\langle \mathcal{R}_s(S_{H\tau},P_{H\tau}),\Psi\rangle  &:= \int_{0}^T \langle \Phi^0 \partial_t S_{H\tau}(\cdot,t), \Psi(\cdot,t ) \rangle_{H^{-1}(\Omega),H^1_0(\Omega)} \hspace{2pt} dt\\
&\enspace \hspace{-50pt} + \int_{\Omega_T} K^0 (\lambda_{w}(S_{H\tau}) \nabla P_{H\tau} + \nabla \Upsilon(S_{H\tau}) - \lambda_w(S_{H\tau}) \rho_w g) \cdot \nabla \Psi.
\end{align*}
for $\Psi \in L^2((0,T),H^1_0(\Omega))$. It is obvious that $(S_{H\tau},P_{H\tau}) \in \Epsilon$ with $S_{H\tau}(\cdot,0)=S_0$ is the weak solution of (\ref{weak-hom-form-1})-(\ref{weak-hom-form-2}) if and only if $\langle \mathcal{R}_p(S_{H\tau},P_{H\tau}),\Psi \rangle=\langle \mathcal{R}_s(S_{H\tau},P_{H\tau}),\Psi\rangle=0$.

In the next step, the errors can be bounded by the residuals in analogy to \cite[Theorem 5.7]{Cances:Pop:Vohralik:2012} (the only difference is the presence of gravity, which however does not have a big influence and the proof still goes through). This leads to the following result.
\begin{lemma}[Upper bound on the error by the residuals]
\label{upper-bounds-on-error-by-resudals}
Let assumptions (A1)-(A9) and (A16) be fulfilled and let $(S_{H\tau},P_{H\tau}) \in \Epsilon$ denote an arbitrary pair of functions. Then the following estimate holds true:
\begin{eqnarray*}
\lefteqn{ ||| S_{H\tau} - S^0 |||^2_{\Omega_T} + \| P_{H\tau} -  P^0 \|_{E(\Omega_T)}^2 + \| \Upsilon(S_{H\tau}) - \Upsilon(S^0) \|_{L^2(\Omega_T)}^2}\\
&\lesssim |||S_{H\tau}(\cdot,0) - S_0 |||^2_{\Omega} + |||\mathcal{R}_p(S_{H\tau},P_{H\tau})|||_{\Omega_T}^2 + |||\mathcal{R}_s(S_{H\tau},P_{H\tau})|||_{\Omega_T}^2.
\end{eqnarray*}
\end{lemma}
To finalize proof of the a posteriori error estimates stated in Theorem \ref{a-posteriori-main-result} and Corollary \ref{a-posteriori-main-result-periodic}, we hence only need to bound the residuals. Again, this is analogous to \cite{Cances:Pop:Vohralik:2012}, with the only difference, that it also involves an estimate for the effective intrinsic permeability. But this estimate for $K^0_{\epsilon,\kappa}$ can be reduced to an a posteriori error estimate for standard elliptic problems. Hence in this setting, the subsequent lemma is easy to prove.

\begin{lemma}
\label{estimate-fine-scale-error}Let assumption (A2) be fulfilled, let $\D \in \mathcal{D}_H^n$ and let $x_{\D}$ denote the corresponding center. Then there holds
\begin{eqnarray*}
\lefteqn{\| K^0_{\epsilon,\kappa} - K^0_{\epsilon,\kappa,h} \|_{L^{\infty}(\D)}
\le C \frac{\beta_{\mathcal{D},\kappa}}{\alpha_{\mathcal{D},\kappa}} \sup_{x \in \D} \|(K_{\epsilon,\kappa}(x,\cdot)-K_{\epsilon,\kappa,h}(x_{\D},\cdot)) (e_i + \nablay w_{\kappa,h}^i(x_{\D},\cdot))\|_{L^2(Y)} \qquad }\quad\\
&\enspace& \quad + C \frac{\beta_{\mathcal{D},\kappa}}{\alpha_{\mathcal{D},\kappa}} \left( \sum_{E \in \Gamma(\Tauh(Y))} h_E \| [K_{\epsilon,\kappa,h}(x_{\D},\cdot) (e_i + \nablay w_{\kappa,h}^i(x_{\D},\cdot))]_{E} \|_{L^2(E)}^2 \right)^{\frac{1}{2}},
\end{eqnarray*}
where $C$ is a generic constant and $\alpha_{\mathcal{D},\kappa}$ denotes the smallest and $\beta_{\mathcal{D},\kappa}$ the largest eigenvalue of $K^{\epsilon}(x_{\D}+ \kappa y)$ in $Y$.
\end{lemma}
Combining Lemmas \ref{upper-bounds-on-error-by-resudals} and \ref{estimate-fine-scale-error} with the previous results concludes the proof of Theorem \ref{a-posteriori-main-result} and Corollary \ref{a-posteriori-main-result-periodic}.

\section{Conclusion}
In this paper, we derived the first version of a heterogeneous multiscale method for the incompressible two-phase flow equations. The method can be applied to the classical formulation of the two-phase flow system as well as to the Kirchhoff transformed system. We presented different discretization strategies that are based on a finite element 
method on the micro scale and a vertex centered finite volume method on the macro scale. In the periodic setting, the method is equivalent to a discretization of the homogenized equation. Finally, a rigorous and effective a posteriori error estimate for the error between HMM approximations and the homogenized solutions is presented. The estimate is based on the results obtained in \cite{Cances:Pop:Vohralik:2012}.

$\\$
{\bf{Acknowledgements.}} We would like to thank the anonymous reviewers for their accurate and constructive feedback which helped us to improve the paper.

\thispagestyle{myheadings}\markright{}
\bibliographystyle{abbrv}


\end{document}